\documentclass[11pt]{article}
\usepackage{pb-diagram}
\usepackage{amsmath,amsthm,amsfonts,amssymb,latexsym}
\usepackage{amsmath}

\usepackage{a4wide}
\newtheorem{thm}{Theorem}[section]
\newtheorem{cor}[thm]{Corollary}
\newtheorem{lem}[thm]{Lemma}
\newtheorem{prop}[thm]{Proposition}
\newtheorem{ex}[thm]{Example}
\newtheorem{defn}[thm]{Definition}
\newtheorem{defns}[thm]{Definitions}
\newtheorem{rmk}[thm]{Remark}
\newtheorem{rems}[thm]{Remarks}
\newcommand{\ad}{\mathrm{ad}}

\begin{document}

%-----------------------------------------------------------------------------------

\title{CONSTRUCTIONS AND $T^{\ast}$-EXTENSIONS OF 3-BIHOM-LIE SUPERALGEBRAS}

%\label{firstpage}

\author{ISMAIL LARAIEDH\thanks{D\'epartement de Math\'ematiques,
Facult\'e des Sciences de Sfax, BP 802, 3038 Sfax, Tunisie. E.mail:
Ismail.laraiedh@gmail.com and Departement of Mathematics, College of Sciences and Humanities - Kowaiyia, Shaqra University,
Kingdom of Saudi Arabia. E.mail:
ismail.laraiedh@su.edu.sa}}

%--------------------------------------------------------------------
\maketitle
% ----------------------------------------------------------------

\begin{abstract}
The aim of this paper is to generalise the construction of $3$-Bihom-Lie superalgebras and we provide
some properties can be lifted to its $T^{\ast}$-extensions such as nilpotency, solvability and decomposition. We study the representations, $T_{\theta}$-extensions
and $T^{\ast}_{\theta}$-extension of $3$-Bihom-Lie superalgebras and prove the necessary and sufficient conditions for
a $2n$-dimensional quadratic $3$-Bihom-Lie superalgebra to be isomorphic to a $T^{\ast}_{\theta}$-extension.

%deformations.

\end{abstract}
\maketitle {\bf Mathematics Subject Classification} (2010). 17A40, 17A45, 17B10, 17B70.


\bigskip
\thispagestyle{empty}

%%%%%%%%%%%%%%%%%%%%%%%%%%%%%%%%%%%%%%%%%%%%%%%%%%%%%%%%%%%%%%%%%%%%%%%%%%%%%%%
%%%%%%%%%%%%%%%%%%%%%%%%%%%%%%%%%%%%%%%%%%%%%%%%%%%%%%%%%%%%%%%%%%%%%%%%%%%%%%%
\section{INTRODUCTION}
%%%%%%%%%%%%%%%%%%%%%%%%%%%%%%%%%%%%%%%%%%%%%%%%%%%%%%%%%%%%%%%%%%%%%%%%%%%%%%%
%%%%%%%%%%%%%%%%%%%%%%%%%%%%%%%%%%%%%%%%%%%%%%%%%%%%%%%%%%%%%%%%%%%%%%%%%%%%%%%*
The origin of Hom-structures can be found in the physics literature around 1900, appearing
in the study of quasideformations of Lie algebras of vector fields. n-ary Hom-type generalization of n-ary algebras
were introduced in \cite{H L S}. Derivations and generalized derivations of many varieties of
algebras and Hom-algebras were investigated in \cite{A M A, K2, L L, K N, Z K, Z C M, Z Z}. Then, hom-Lie algebras were generalized to hom-Lie superalgebras by Ammar
and Makhlouf \cite{F Ab, F Ab N}.

A Bihom-algebra is an algebra in such a way that the identities defining the structure
are twisted by two homomorphisms. This class of algebras was introduced from
a categorical approach in \cite{G M M P} as an extension of the class of Hom-algebras. When the
two linear maps are same, then Bihom-algebras will be return to Hom-algebras. The representation theory of Bihom-Lie algebras was introduced by Cheng in
\cite{C Q}, in which, Bihom-cochain complexes, derivation, central extension, derivation extension,
trivial representation and adjoint representation of Bihom-Lie algebras were studied. In particular, the definition of $n$-Bihom-Lie algebras was introduced in [11, 5, 6].

In 1997 Bordermann introduced the notion of $T^{\ast}$-extension of Lie algebras \cite{B}, which
is a workable extensional technique since it is a one-step procedure: it is one of the main
tools to prove that every symplectic quadratic Lie algebra is a special symplectic Manin
algebra \cite{B B M}. Many facts show that $T^{\ast}$-extension is an important method to
study algebraic structures \cite{B, B B M, L C S, L C M, Li C}.

In the present article, we study the $3-$Bihom-Lie superalgebras, which can
be viewed as an extension of $3-$Bihom-Lie algebras to $\mathbb{Z}_{2}$-graded algebras. Recently, the definition of Bihom-Lie superalgebras were introduced in \cite{W G}. More applications of the Bihom-Lie superalgebras, $3-$Bihom-Lie superalgebras can be found in
\cite{A S O, S S}

The paper proceds as follows. In Section 2, summarizes basic concepts and recall the definition of $3$-Bihom-Lie superalgebras, and show that a $3-$Bihom-Lie superalgebra is given by the direct sum of two
$3$-Bihom-Lie superalgebras and the tensor product of a $3$-totally Bihom-associative superalgebra
and a $3$-Bihom-Lie superalgebra. Also we prove that a homomorphism between $3$-Bihom-Lie
superalgebras is a morphism if and only if its graph is a $3-$Bihom subalgebra. In Section 3,
we give the definition of representations of $3$-Bihom-Lie superalgebras. We can obtain the semidirect product $3$-Bihom-Lie superalgebra $(\mathfrak{g}\oplus M, [\cdot, \cdot, \cdot]_{\rho}, \alpha+\alpha_M,\beta+\beta_M)$ associated with any representation $\rho$ of a $3$-Bihom-Lie superalgebra $(\mathfrak{g}, [\cdot, \cdot, \cdot], \alpha, \beta)$ on $M$. And we can get a $T_\theta$-extension of $(\mathfrak{g}, [\cdot, \cdot, \cdot], \alpha, \beta)$ by a $3$-cocycle $\theta$. In Section 4, $T_\theta^*$-extensions of $3$-Bihom-Lie superalgebras are studied. We give the necessary and sufficient
conditions for a $2n$-dimensional quadratic $3$-Bihom-Lie superalgebra to be isomorphic to a $T_\theta^*$-extension.

%%%%%%%%%%%%%%%%%%%%%%%%%%%%%%%%%%%%%%%%%%%%%%%%%%%%%%%%%%%%%%%%%%
\section{DEFINITIONS AND PROPRIETIES}
%%%%%%%%%%%%%%%%%%%%%%%%%%%%%%%%%%%%%%%%%%%%%%%%%%%%%%%%%%
\begin{defn}\cite{W G}
A Bihom-Lie superalgebra over a field $\mathbb{K}$ is a $4$-tuple $(\mathfrak g, [�\cdot, \cdot], \alpha,\beta)$ consisting
of a $\mathbb{Z}_2$-graded vector space $\mathfrak g=\mathfrak{g}_{0}\oplus \mathfrak{g}_{1}$, an even bilinear map $[\cdot,\cdot] : \mathfrak g\times\mathfrak g\longrightarrow\mathfrak g$ and two even endomorphisms $\alpha,\beta: \mathfrak g \rightarrow \mathfrak g$ satisfying the following
identities, $\forall x,y,z,u,v\in\mathfrak{g}$
\begin{enumerate}
\item
$\alpha\circ\beta=\beta\circ\alpha,$\\
\item
$\alpha([x,y])=[\alpha(x),\alpha(y)],\ \beta([x,y])=[\beta(x),\beta(y)],$\\
\item
$[\beta(x),\alpha(y)] = -(-1)^{|x||y|}[\beta(y), \alpha(x)],$\\
\item
$\displaystyle\circlearrowleft_{x,y,z}(-1)^{|x||z|}[\beta^2(x), [\beta(y),\alpha(z)]] = 0$.
\end{enumerate}
\end{defn}
\begin{thm}\label{induced-color}
Let $(\mathfrak{g},[.,.])$ be a Lie superalgebra and $\alpha,\beta:\mathfrak{g}\longrightarrow\mathfrak{g}$
be two algebra homomorphisms
such that $\alpha\circ\beta=\beta\circ\alpha$, then $(\mathfrak{g},[.,.]_{\alpha,\beta},\alpha,\beta)$ is a Bihom-Lie superalgebra, where
$[x,y]_{\alpha,\beta}=[\alpha(x),\beta(y)].$
\end{thm}
\begin{proof} Obviously $[.,.]_{\alpha,\beta}$ is Bihom-superskewsymmetric. Furthermore $(\mathfrak{g},[.,.]_{\alpha,\beta},\alpha,\beta)$ satisfies the Bihom-super-Jacobi condition. Indeed  %$\eqref{HJCI}$
\begin{eqnarray*}
   \circlearrowleft_{x,y,z}(-1)^{|z||x|}[\beta^2(x),[\beta(y),\alpha(z)]_{\alpha,\beta}]_{\alpha,\beta} &=&
   \circlearrowleft_{x,y,z}(-1)^{|z||x|}[\beta^2(x),[\alpha\beta(y),\beta\alpha(z)]]_{\alpha,\beta} \\
   &=& \circlearrowleft_{x,y,z}(-1)^{|z||x|}[\alpha\beta^2(x),[\alpha\beta^2(y),\alpha\beta^2(z)]] \\
 &=& \alpha\beta^2\circlearrowleft_{x,y,z}(-1)^{|z||x|}[x,[y,z]] \\
   &=&  0.
\end{eqnarray*}
\end{proof}
%\begin{exa}(Heisenberg BiHom-Lie colour algebra) Let $H$ be the three dimensional
%Heisenberg Lie algebra, which consists of the strictly upper-triangular complex $3\times 3$ matrices.
%It has a standard linear basis:
%\begin{eqnarray*}
%% \nonumber to remove numbering (before each equation)
%&& e_1=\left(
%  \begin{array}{ccc}
%    0 & 1 & 0 \\
%    0 & 0 & 0 \\
%    0 & 0 & 0 \\
%  \end{array}
%\right),~~e_2=\left(
%  \begin{array}{ccc}
%    0 & 0 & 0 \\
%    0 & 0 & 1 \\
%    0 & 0 & 0 \\
%  \end{array}
%\right),~~e_3=\left(
%  \begin{array}{ccc}
%    0 & 0 & 1 \\
%    0 & 0 & 0 \\
%    0 & 0 & 0 \\
%  \end{array}
%\right)
%\end{eqnarray*}
%satisfying
%\begin{eqnarray*}
%% \nonumber to remove numbering (before each equation)
%&& [e_1,e_2]=e_3,~~[e_1,e_3]=[e_2,e_3]=0.
%\end{eqnarray*}
%Let $\Gamma=\mathbb{Z}_2\times \mathbb{Z}_2\times \mathbb{Z}_2$ with a bicharacter $\varepsilon$ given by $\varepsilon(i,j)=(-1)^{i_1j_1+i_2j_2+i_3j_3}$ for all $i=(i_1,i_2,i_3)$ and $j=(j_1,j_2,j_3)$ in $\Gamma$
%\end{exa}
\begin{ex}\cite{Ka Ab} We consider in the sequel the matrix realization
of this Lie superalgebra.\\
 Let $osp(1, 2) = \mathfrak{g}_0 \oplus \mathfrak{g}_1$ be the Lie superalgebra where $\mathfrak{g}_0$ is spanned by
\begin{eqnarray*}
&& H=\left(
  \begin{array}{ccc}
    1 & 0 & 0 \\
    0 & 0 & 0 \\
    0 & 0 & -1 \\
  \end{array}
\right),~~X=\left(
  \begin{array}{ccc}
    0 & 0 & 1 \\
    0 & 0 & 0 \\
    0 & 0 & 0 \\
  \end{array}
\right),~~Y=\left(
  \begin{array}{ccc}
    0 & 0 & 0 \\
    0 & 0 & 0 \\
    1 & 0 & 0 \\
  \end{array}
\right),
\end{eqnarray*}
and $\mathfrak{g}_1$ is spanned by
\begin{eqnarray*}
&& F=\left(
  \begin{array}{ccc}
    0 & 0 & 0 \\
    1 & 0 & 0 \\
    0 & 1 & 0 \\
  \end{array}
\right),~~G=\left(
  \begin{array}{ccc}
    0 & 1 & 0 \\
    0 & 0 & -1 \\
    0 & 0 & 0 \\
  \end{array}
\right).
\end{eqnarray*}
The defining relations $($we give only the ones with non-zero values in the right-hand side$)$ are
\begin{eqnarray*}
&& [H,X]= 2X,\hskip1cm[H,Y] = -2Y,\hskip1cm [X,Y] = H,\\
&& [Y,G]= F,\hskip1cm[X,F] = G,\hskip1cm [H,F]=-F,\hskip1cm [H,G] = G,\\
&& [G,F]= H,\hskip1cm [G,G]=-2X,\hskip1cm [F,F]= 2Y.
\end{eqnarray*}
Let $\lambda,\mu \in \mathbb{R}^\ast$, we consider the linear maps $\alpha_\lambda: osp(1,2)\longrightarrow osp(1,2)$ and $\beta_\mu: osp(1,2)\longrightarrow osp(1,2)$ defined by
\begin{eqnarray*}
\alpha_\lambda(X)&=& \lambda^2 X,~~\alpha_\lambda(Y)=\frac{1}{\lambda^2}Y,~~\alpha_\lambda(H)=H,~~\alpha_\lambda(F)=\frac{1}{\lambda}F,~~ \alpha_\lambda(G)=\lambda G,\\
\beta_\mu(X)&=& \mu^2 X,~~\beta_\mu(Y)=\frac{1}{\mu^2}Y,~~\beta_\mu(H)=H,~~\beta_\mu(F)=\frac{1}{\mu}F,~~ \beta_\mu(G)=\mu G.
\end{eqnarray*}
Obviously, we have $\alpha_\lambda \circ \beta_\mu=\beta_\mu \circ \alpha_\lambda$. For all $H,X,Y,F$ and $G$ in $osp(1,2)$, we have
\begin{eqnarray*}
\alpha_\lambda([H,X])&=& \alpha_\lambda(2X)=2 \lambda^2 X,~~
\alpha_\lambda([H,Y])= \alpha_\lambda(-2Y)=-2\frac{1}{\lambda^2}Y,\\
\alpha_\lambda([X,Y])&=& \alpha_\lambda(H)=H,~~
\alpha_\lambda([Y,G])= \alpha_\lambda(F)=\frac{1}{\lambda}F,\\
\alpha_\lambda([X,F])&=& \alpha_\lambda(G)=\lambda G,~~
\alpha_\lambda([H,F])= \alpha_\lambda(-F)=-\frac{1}{\lambda}F,\\
\alpha_\lambda([H,G])&=& \alpha_\lambda(G)=\lambda G,~~
\alpha_\lambda([G,F])= \alpha_\lambda(H)=H,\\
\alpha_\lambda([G,G])&=& \alpha_\lambda(-2X)=-2\lambda^2 X,~~
\alpha_\lambda([F,F])= \alpha_\lambda(2Y)=2\frac{1}{\lambda^2}Y.
\end{eqnarray*}
On the other hand, we have
\begin{eqnarray*}
&&[\alpha_\lambda(H),\alpha_\lambda(X)]=[H,\lambda^2 X]=2\lambda^2 X,~~
[\alpha_\lambda(H),\alpha_\lambda(Y)]= [H,\frac{1}{\lambda^2}Y]=-2\frac{1}{\lambda^2}Y,\\
&&[\alpha_\lambda(X),\alpha_\lambda(Y)]=[\lambda^2 X,\frac{1}{\lambda^2}Y]=H,~~
[\alpha_\lambda(Y),\alpha_\lambda(G)]=[\frac{1}{\lambda^2}Y,\lambda G]=\frac{1}{\lambda}F,\\
&&[\alpha_\lambda(X),\alpha_\lambda(F)]=[\lambda^2 X,\frac{1}{\lambda}F]=\lambda G,~~
[\alpha_\lambda(H),\alpha_\lambda(F)]=[H,\frac{1}{\lambda}F]=-\frac{1}{\lambda}F,\\
&&[\alpha_\lambda(H),\alpha_\lambda(G)]=[H,\lambda G]=\lambda G,~~
[\alpha_\lambda(G),\alpha_\lambda(F)]= [\lambda G,\frac{1}{\lambda}F]=H,\\
&&[\alpha_\lambda(G),\alpha_\lambda(G)]=[\lambda G,\lambda G]=-2\lambda^2 X,~~
[\alpha_\lambda(F),\alpha_\lambda(F)]= [\frac{1}{\lambda}F,\frac{1}{\lambda}F]=2\frac{1}{\lambda^2}Y.
\end{eqnarray*}
Therefore, for $a,a^{'}\in osp(1,2)$, we have
\begin{eqnarray*}
\alpha_\lambda([a,a^{'}])&=& [\alpha_\lambda(a),\alpha_\lambda(a^{'})].
\end{eqnarray*}
Similarly, we have
\begin{eqnarray*}
\beta_\mu([a,a^{'}])&=& [\beta_\mu(a),\beta_\mu(a^{'})].
\end{eqnarray*}
Applying Theorem \ref{induced-color} we obtain a family of Bihom-Lie superalgebras\\ $osp(1,2)_{\alpha_\lambda,\beta_\mu} = \Big(osp(1,2),[.,.]_{\alpha_{\lambda},\beta_{\mu}}=[\cdot,\cdot]\circ (\alpha_\lambda \otimes \beta_\mu),\alpha_\lambda,\beta_\mu \Big)$ where the Bihom-Lie superalgebra bracket $[.,.]_{\alpha_{\lambda},\beta_{\mu}}$
 on the basis elements is given, for $\lambda,\mu \neq 0$, by
\begin{eqnarray*}
&& [H,X]_{\alpha_{\lambda},\beta_{\mu}}= 2\mu^{2}X,\hskip1cm [H,Y]_{\alpha_{\lambda},\beta_{\mu}} = \frac{-2}{\mu^{2}}Y,\hskip1cm [X,Y]_{\alpha_{\lambda},\beta_{\mu}} =(\frac{\lambda}{\mu})^2 H,\\
&& [Y,G]_{\alpha_{\lambda},\beta_{\mu}}=\frac{\mu}{\lambda^2} F,\hskip1cm [X,F]_{\alpha_{\lambda},\beta_{\mu}} =\frac{\lambda^2}{\mu} G,\hskip1cm [H,F]_{\alpha_{\lambda},\beta_{\mu}}=-\frac{1}{\mu}F,\\
&& [G,F]_{\alpha_{\lambda},\beta_{\mu}}=\frac{\lambda}{\mu} H,\hskip1cm [G,G]_{\alpha_{\lambda},\beta_{\mu}}=-2\lambda \mu X,\hskip1cm[F,F]_{\alpha_{\lambda},\beta_{\mu}}= 2\frac{\lambda}{\mu}Y,\\
&& [H,G]_{\alpha_{\lambda},\beta_{\mu}} =\mu G.
\end{eqnarray*}
\end{ex}

\begin{defn}\cite{A S O}
A $3$-Bihom Lie superalgebra over a field $\mathbb{K}$ is a quadruple $(\mathfrak{g},[.,.,.],\alpha,\beta),$ where consisting of a $\mathbb{Z}_{2}$-graded vector space
$\mathfrak{g}=\mathfrak{g}_{0}\oplus\mathfrak{g}_{1}$, an even trilinear map $[.,.,.]: \mathfrak{g}\otimes\mathfrak{g}\otimes\mathfrak{g}\longrightarrow\mathfrak{g}$ and two even endomorphisms $\alpha,\beta:\mathfrak{g}\longrightarrow\mathfrak{g}$ satisfying the following conditions, $\forall x,y,z,u,v\in\mathfrak{g}$
\begin{enumerate}
\item
$\alpha\circ\beta=\beta\circ\alpha,$
\item
$\alpha([x,y,z]_)=[\alpha(x),\alpha(y),\alpha(z)]~and~\beta([x,y,z])=[\beta(x),\beta(y),\beta(z)],$
\item
$[\beta(x),\beta(y),\alpha(z)]=-(-1)^{|x||y|}[\beta(y),\beta(x),\alpha(z)]=-(-1)^{|y||z|}[\beta(x),\beta(z),\alpha(y)],~~$($3-$Bihom-super-skewsymmetry).
\item
 $ $\\
$\begin{array}{lllllll}[\beta^{2}(x),\beta^{2}(y),[\beta(z),\beta(u),\alpha(v)]]&=&(-1)^{(|u|+|v|)(|x|+|y|+|z|)}[\beta^{2}(u),\beta^{2}(v),[\beta(x),\beta(y),\alpha(z)]]\\
&&-(-1)^{(|z|+|v|)(|x|+|y|)+|u||v|}[\beta^{2}(z),\beta^{2}(v),[\beta(x),\beta(y),\alpha(u)]]\\
&&+(-1)^{(|z|+|u|)(|x|+|y|)}[\beta^{2}(z),\beta^{2}(u),[\beta(x),\beta(y),\alpha(v)]],\end{array}$

($3-$Bihom-super-Jacobi identity).

\end{enumerate}
\end{defn}
\begin{rems}
\begin{enumerate}
\item
The $3-$Bihom-super-Jacobi identity is equivalent to
$$[\beta^{2}(x),\beta^{2}(y),[\beta(z),\beta(u),\alpha(v)]]=(-1)^{|z||v|}\circlearrowleft_{u,v,z}(-1)^{\gamma}[\beta^{2}(u),\beta^{2}(v),[\beta(x),\beta(y),\alpha(z)]],$$

where $\gamma=(|u|+|v|)(|x|+|y|)+|z||u|$
\item
A $3-$Lie superalgebra $(\mathfrak{g},[.,.,.])$ is a $3-$BiHom-Lie superalgebra with $\alpha=\beta=Id$, since the $3-$BiHom-super-Jacobi condition
reduces to the $3$-super-Jacobi condition when $\alpha=\beta=Id$.
\end{enumerate}
\end{rems}
\begin{defns}
\begin{enumerate}
\item
A $3$-Bihom Lie superalgebra $(\mathfrak{g},[.,.,.],\alpha,\beta)$ is regular if $\alpha$ and $\beta$ are an algebra automorphisms.
\item
A sub-vector space $\eta\subset \mathfrak{g}$ is a $3-$Bihom subalgebra of $(\mathfrak{g},[.,.,.],\alpha,\beta)$ if $\alpha(\eta)\subset \eta$,\\
$\beta(\eta)\subset\eta$ and $[\eta,\eta,\eta]\subseteq\eta$, (i.e.
$[x,y,z]\in\eta,~~\forall x,y,z\in\eta$).
\item
A sub-vector space $\eta\subset \mathfrak{g}$ is a Bihom ideal of $(\mathfrak{g},[.,.,.],\alpha,\beta)$ if $\alpha(\eta)\subset \eta$, $\beta(\eta)\subset\eta$\\
and $[\eta,\mathfrak{g},\mathfrak{g}]\subseteq\eta$, (i.e.
$[x,y,z]\in\eta,~~\forall x\in\eta;~y,z\in\mathfrak{g})$.
\end{enumerate}
\end{defns}
\begin{defns}
\begin{enumerate}
\item
The center of $(\mathfrak{g},[.,.,.],\alpha,\beta)$ is the set of $x\in\mathfrak{g}$ such that $[x,y_{1},y_{2}]=0$ for any $y_{1},y_{2}\in\mathfrak{g}$. The center is an ideal of $\mathfrak{g}$ which we will denote by $Z(\mathfrak{g})$.
\item
The $(\alpha,\beta)$-center of $(\mathfrak{g},[.,.,.],\alpha,\beta)$ is the set
$$Z({\alpha,\beta})=\{x\in\mathfrak{g},~[x,\alpha\beta(y_{1}),\alpha\beta(y_{2})]=0,~~for~any~y_{1},y_{2}\in\mathfrak{g}\}$$
\end{enumerate}

\end{defns}
\begin{prop}
Let $(\mathfrak{g},[.,.,.])$ be a $3$-Lie superalgebra and $\alpha,\beta:\mathfrak{g}\longrightarrow\mathfrak{g}$
be two algebra homomorphisms
such that $\alpha\circ\beta=\beta\circ\alpha$, then $(\mathfrak{g},[.,.,.]_{\alpha,\beta},\alpha,\beta)$ is a $3$-Bihom-Lie superalgebra, where
$[x,y,z]_{\alpha,\beta}=[\alpha(x),\alpha(y),\beta(z)]].$
\end{prop}
\begin{proof}
It is easy to see that, $\forall x,y,z\in\mathfrak{g}$ we have

$\alpha([x,y,z]_{\alpha,\beta})=[\alpha(x),\alpha(y),\alpha(z)]_{\alpha,\beta},$

$\beta([x,y,z]_{\alpha,\beta})=[\beta(x),\beta(y),\beta(z)]_{\alpha,\beta},$\\[0.2cm]
and\\

$[\beta(x),\beta(y),\alpha(z)]_{\alpha,\beta}=-(-1)^{|x||y|}[\beta(y),\beta(x),\alpha(z)]_{\alpha,\beta},$

$[\beta(x),\beta(y),\alpha(z)]_{\alpha,\beta}=-(-1)^{|y||z|}[\beta(x),\beta(z),\alpha(y)]_{\alpha,\beta}.$\\

Now, we prove the $3-$Bihom-super-Jacobi identity. Let $x,y,z,u,v\in\mathfrak{g}$,
$$\begin{array}{lllllll}[\beta^{2}(x),\beta^{2}(y),[\beta(z),\beta(u),\alpha(v)]_{\alpha,\beta}]_{\alpha,\beta}
&=&[\beta^{2}(x),\beta^{2}(y),[\alpha\beta(z),\alpha\beta(u),\alpha\beta(v)]]_{\alpha,\beta}\\
&=&[\alpha\beta^{2}(x),\alpha\beta^{2}(y),[\alpha\beta^{2}(z),\alpha\beta^{2}(u),\alpha\beta^{2}(v)]]\\
&=&\alpha\beta^{2}([x,y,[z,u,v]])\\
&=&\alpha\beta^{2}(-1)^{|z||v|}\circlearrowleft_{u,v,z}(-1)^{\gamma}[u,v,[x,y,z]]\\
&=&(-1)^{|z||v|}\circlearrowleft_{u,v,z}(-1)^{\gamma}[\beta^{2}(u),\beta^{2}(v),[\beta(x),\beta(y),\alpha(z)]_{\alpha,\beta}]_{\alpha,\beta},
\end{array}$$
where $\gamma=(|u|+|v|)(|x|+|y|)+|z||u|$.

Then the $3-$Bihom-super-Jacobi identity it satisfies.
\end{proof}
\begin{prop}\label{prop1}
Let $(\mathfrak{g},[.,.,.],\alpha,\beta)$ be a $3$-Bihom-Lie superalgebra, $\alpha',\beta':\mathfrak{g}\longrightarrow\mathfrak{g}$ be two even commuting algebra homomorphisms
and any two of the maps $\alpha,\beta,\alpha',\beta'$ commute. Then $(\mathfrak{g},[.,.,.]_{\alpha',\beta'}:=[.,.,.]\circ(\alpha'\otimes\alpha'\otimes\beta'),\alpha\circ\alpha',\beta\circ\beta')$ is a $3$-Bihom-Lie superalgebra.
\end{prop}
\begin{proof}
First we prove the $3$-Bihom-super-skewsymmetry. For any $x,y,z\in\mathfrak{g},$ we have
$$\begin{array}{lllllllll}[\beta\circ\beta'(x),\beta\circ\beta'(y),\alpha\circ\alpha'(z)]_{\alpha',\beta'}&=&[\alpha'\circ\beta\circ\beta'(x),\alpha'\circ\beta\circ\beta'(y),\beta'\circ\alpha\circ\alpha'(z)]\\
&=&\alpha'\circ\beta'([\beta(x),\beta(y),\alpha(z)])\\
&=&-(-1)^{|x||y|}\alpha'\circ\beta'([\beta(y),\beta(x),\alpha(z)])\\
&=&-(-1)^{|x||y|}[\alpha'\circ\beta'\circ\beta(y),\alpha'\circ\beta'\circ\beta(x),\alpha'\circ\beta'\circ\alpha(z)]\\
&=&-(-1)^{|x||y|}[\beta\circ\beta'(y),\beta\circ\beta'(x),\alpha\circ\alpha'(z)]_{\alpha',\beta'}\end{array}$$
In the same way, $[\beta\circ\beta'(x),\beta\circ\beta'(y),\alpha\circ\alpha'(z)]_{\alpha',\beta'}=-(-1)^{|y||z|}[\beta\circ\beta'(x),\beta\circ\beta'(z),\alpha\circ\alpha'(y)]_{\alpha',\beta'}$.

Now we prove the $3$-Bihom-super-Jacobi identity. For any $x,y,z,u,v\in\mathfrak{g}$, we have
$$\begin{array}{lllllll}&&[(\beta\circ\beta')^{2}(u),(\beta\circ\beta')^{2}(v),[(\beta\circ\beta')(x),\beta\circ\beta'(y),(\alpha\circ\alpha')(z)]_{\alpha',\beta'}]_{\alpha',\beta'}\\
&=&[(\beta\circ\beta')^{2}(u),(\beta\circ\beta')^{2}(v),\alpha'\circ\beta'([\beta(x),\beta(y),\alpha(z)]]_{\alpha',\beta'}\\
&=&[\alpha'\circ(\beta\circ\beta')^{2}(u),\alpha'\circ(\beta\circ\beta')^{2}(v),\alpha'\circ(\beta')^{2}([\beta(x),\beta(y),\alpha(z)])]\\
&=&\alpha'\circ(\beta')^{2}([\beta^{2}(u),\beta^{2}(v),[\beta(x),\beta(y),\alpha(z)]])\\
&=&(-1)^{|x||z|}\circlearrowleft_{y,z,x}(-1)^{\gamma}\alpha'\circ(\beta')^{2}[\beta^{2}(y),\beta^{2}(z),[\beta(u),\beta(v),\alpha(x)]]\\
&=&(-1)^{|x||z|}\circlearrowleft_{y,z,x}(-1)^{\gamma}[(\beta\circ\beta')^{2}(y),(\beta\circ\beta')^{2}(z),[\beta\circ\beta'(u),\beta\circ\beta'(v),\alpha\circ\alpha'(x)]_{\alpha',\beta'}]_{\alpha',\beta'}\\
\end{array}$$
Thus $(\mathfrak{g},[.,.,.]_{\alpha',\beta'}:=[.,.,.]\circ(\alpha'\otimes\alpha'\otimes\beta'),\alpha\circ\alpha',\beta\circ\beta')$ is a $3$-Bihom-Lie superalgebra.
\end{proof}
\begin{cor}
Let $(\mathfrak{g},[.,.,.],\alpha,\beta)$ a  $3$-Bihom-Lie superalgebra. Then $(\mathfrak{g},[.,.,.]_{k}:=[.,.,.]\circ(\alpha^{k}\otimes\alpha^{k}\otimes\beta^{k}),\alpha^{k+1},\beta^{k+1})$ is a $3$-Bihom-Lie superalgebra.
\end{cor}
\begin{proof}
Apply Proposition \ref{prop1} with $\alpha'=\alpha^{k}$ and $\beta'=\beta^{k}$.
\end{proof}
\begin{defn}
A $3$-Totally Bihom-associative superalgebra is a quadruple $(\mathcal{A},\mu,\alpha,\beta),$ consisting of a linear space $\mathcal{A}$, an even trilinear map $\mu:\mathcal{A}\otimes\mathcal{A}\otimes\mathcal{A}$ and two even linear maps $\alpha,\beta:\mathcal{A}\longrightarrow\mathcal{A}$ satisfying the following conditions, $\forall a_{1},a_{2},a_{3},a_{4},a_{5}\in\mathcal{A}$,
\begin{enumerate}
\item
$\alpha\circ\beta=\beta\circ\alpha$,
\item
$\alpha(\mu(a_{1},a_{2},a_{3}))=\mu(\alpha(a_{1}),\alpha(a_{2}),\alpha(a_{3}))~ and ~\beta(\mu(a_{1},a_{2},a_{3}))=\mu(\beta(a_{1}),\beta(a_{2}),\beta(a_{3}))$
\item
$3$-Total Bihom-associativity: $\forall i,j:~1\leq i,j \leq 3$
\begin{equation}\label{10}\begin{array}{llllllll}
&&\mu(\alpha(a_{1}),\cdots,\alpha(a_{i-1}),\mu(a_{i},\cdots,a_{2+i}),\beta(a_{3+i}),\cdots,\beta(a_{5}))\\
&=&\mu(\alpha(a_{1}),\cdots,\alpha(a_{j-1}),\mu(a_{j},\cdots,a_{2+j}),\beta(a_{3+j}),\cdots,\beta(a_{5}))\end{array}\end{equation}
\end{enumerate}
\end{defn}
\begin{rmk}
The equation(\ref{10}) is equivalent to
$$\mu(\mu(a_{1},a_{2},a_{3}),\beta(a_{4}),\beta(a_{5}))=\mu(\alpha(a_{1}),\mu(a_{2},a_{3},a_{4}),\beta(a_{5}))=\mu(\alpha(a_{1}),\alpha(a_{2}),\mu(a_{3},a_{4},a_{5}))$$
\end{rmk}
\begin{defn}
A $3$-Partially Bihom-associative superalgebra is a quadruple $(\mathcal{A},\mu,\alpha,\beta),$ consisting of a linear space $\mathcal{A}$, an even trilinear map $\mu:\mathcal{A}\otimes\mathcal{A}\otimes\mathcal{A}$ and two even linear maps $\alpha,\beta:\mathcal{A}\longrightarrow\mathcal{A}$ satisfying the following conditions, $\forall a_{1},a_{2},a_{3},a_{4},a_{5}\in\mathcal{A}$,
\begin{enumerate}
\item
$\alpha\circ\beta=\beta\circ\alpha$,
\item
$\alpha(\mu(a_{1},a_{2},a_{3}))=\mu(\alpha(a_{1}),\alpha(a_{2}),\alpha(a_{3}))~ and ~\alpha(\mu(a_{1},a_{2},a_{3}))=\mu(\alpha(a_{1}),\alpha(a_{2}),\alpha(a_{3}))$
\item
$3$-Partial Bihom-associativity: $\forall i,j:~1\leq i,j \leq 3$
\begin{equation}
\sum_{i=1}^{3}\mu(\alpha(a_{1}),\cdots,\alpha(a_{i-1}),\mu(a_{i},\cdots,a_{2+i}),\beta(a_{3+i}),\cdots,\beta(a_{5}))=0
\end{equation}
\end{enumerate}
\end{defn}
\begin{prop}
Let $(\mathcal{A},\mu,\alpha_{1},\beta_{1})$ be a $3$-totally Bihom-associative superalgebra and $(\mathfrak{g},[.,.,.],\alpha_{2},\beta_{2})$ a $3$-Bihom-Lie superalgebra.
If $\alpha_{1}$ is surjective and $\forall a_{1},a_{2},a_{3}\in\mathcal{A},$
\begin{equation}\label{eq}\mu(\beta_{1}(a_{1}),\beta_{1}(a_{2}),\alpha_{1}(a_{3}))=\mu(\beta_{1}(a_{2}),\beta_{1}(a_{1}),\alpha_{1}(a_{3}))=\mu(\beta_{1}(a_{1}),\beta_{1}(a_{3}),\alpha_{1}(a_{2})).\end{equation}
Then $(\mathcal{A}\otimes\mathfrak{g},[.,.,.]_{\mathcal{A}\otimes\mathfrak{g}},\alpha,\beta)$ is a $3$-Bihom-Lie superalgebra, where the trilinear map $[.,.,.]_{\mathcal{A}\otimes\mathfrak{g}}:\wedge^{3}(\mathcal{A}\otimes\mathfrak{g})\longrightarrow\mathcal{A}\otimes\mathfrak{g}$ are given by $$[a_{1}\otimes x_{1},a_{2}\otimes x_{2},a_{3}\otimes x_{3}]_{\mathcal{A}\otimes\mathfrak{g}}=\mu(a_{1},a_{2},a_{3})\otimes[x_{1},x_{2},x_{3}],~\forall a_{i}\in\mathcal{A},~x_{i}\in\mathfrak{g},~i=1,2,3,$$
and the two linear maps $\alpha,\beta:\mathcal{A}\otimes\mathfrak{g}\longrightarrow\mathcal{A}\otimes\mathfrak{g}$ are given by $\alpha(a_{1}\otimes x_{1})=\alpha_{1}(a_{1})\otimes \alpha_{2}(x_{1})$ and $\beta(a_{1}\otimes x_{1})=\beta_{1}(a_{1})\otimes \beta_{2}(x_{1})$.
\end{prop}
\begin{proof}
Since $\alpha_{1}\circ\beta_{1}=\beta_{1}\circ\alpha_{1}$ and $\alpha_{2}\circ\beta_{2}=\beta_{2}\circ\alpha_{2}$, we have $\alpha\circ\beta=\beta\circ\alpha$.

First we prove $[.,.,.]_{\mathcal{A}\otimes\mathfrak{g}}$ satisfies Bihom-super-skewsymmetry. For $a_{i}\in\mathcal{A},~x_{i}\in\mathfrak{g},~i=1,2,3,$ we have
$$\begin{array}{llllllll}
&&[\beta(a_{1}\otimes x_{1}),\beta(a_{2}\otimes x_{2}),\alpha(a_{3}\otimes x_{3})]_{\mathcal{A}\otimes\mathfrak{g}}\\
&=&[\beta_{1}(a_{1})\otimes \beta_{2}(x_{1}),\beta_{1}(a_{2})\otimes \beta_{2}(x_{2}),\alpha_{1}(a_{3})\otimes \alpha_{2}(x_{3})]_{\mathcal{A}\otimes\mathfrak{g}}\\
&=&\mu(\beta_{1}(a_{1}),\beta_{1}(a_{2}),\alpha_{1}(a_{3}))[\beta_{2}(x_{1}),\beta_{2}(x_{2}),\alpha_{2}(x_{3})]\\
&=&-(-1)^{|x_{1}||x_{2}|}\mu(\beta_{1}(a_{2}),\beta_{1}(a_{1}),\alpha_{1}(a_{3}))[\beta_{2}(x_{2}),\beta_{2}(x_{1}),\alpha_{2}(x_{3})]\\
&=&-(-1)^{|x_{1}||x_{2}|}[\beta(a_{2}\otimes x_{2}),\beta(a_{1}\otimes x_{1}),\alpha(a_{3}\otimes x_{3})]_{\mathcal{A}\otimes\mathfrak{g}}
\end{array}$$
In the same way,
$$[\beta(a_{1}\otimes x_{1}),\beta(a_{2}\otimes x_{2}),\alpha(a_{3}\otimes x_{3})]_{\mathcal{A}\otimes\mathfrak{g}}=-(-1)^{|x_{2}||x_{3}|}[\beta(a_{1}\otimes x_{1}),\beta(a_{3}\otimes x_{3}),\alpha(a_{2}\otimes x_{2})]_{\mathcal{A}\otimes\mathfrak{g}}$$

Now, we prove the $3-$Bihom-super-Jacobi identity. Let $a_{i}\in\mathcal{A},~x_{i}\in\mathfrak{g},~i=1,2,3,$
$$\begin{array}{llllllllll}
&&(-1)^{(|x_{4}|+|x_{5}|)(|x_{1}|+|x_{2}|+|x_{3}|)}\times\\&&[\beta^2(a_4\otimes x_4), \beta^2(a_5\otimes x_5), [\beta(a_1\otimes x_1), \beta(a_2\otimes x_2), \alpha(a_3\otimes x_3)]_{A\otimes \mathfrak{g}}]_{A\otimes \mathfrak{g}}\\
&&-(-1)^{(|x_{3}|+|x_{5}|)(|x_{1}|+|x_{2}|)+|x_{4}||x_{5}|}\times\\&&[\beta^2(a_3\otimes x_3), \beta^2(a_5\otimes x_5), [\beta(a_1\otimes x_1), \beta(a_2\otimes x_2), \alpha(a_4\otimes x_4)]_{A\otimes \mathfrak{g}}]_{A\otimes \mathfrak{g}}\\
&&+(-1)^{(|x_{3}|+|x_{4}|)(|x_{1}|+|x_{2}|)}\times\\&&[\beta^2(a_3\otimes x_3), \beta^2(a_4\otimes x_4), [\beta(a_1\otimes x_1), \beta(a_2\otimes x_2), \alpha(a_5\otimes x_5)]_{A\otimes L}]_{A\otimes \mathfrak{g}}\\
&=&(-1)^{(|x_{4}|+|x_{5}|)(|x_{1}|+|x_{2}|+|x_{3}|)}\times\\&&\mu(\beta_{1}^{2}(a_4),\beta_{1}^{2}(a_5),\mu\big(\beta_{1}(a_1),\beta_{1}(a_2),\alpha_{1}(a_3)\big)\big)\otimes[\beta_{2}^{2}(x_4), \beta_{2}^{2}(x_5), [\beta_{2}(x_1), \beta_{2}(x_2), \alpha_{2}(x_3)]]\\
&&-(-1)^{(|x_{3}|+|x_{5}|)(|x_{1}|+|x_{2}|)+|x_{4}||x_{5}|}\times\\&&\mu\big(\beta_{1}^{2}(a_3),\beta_{1}^{2}(a_5),\mu\big(\beta_{1}(a_1),\beta_{1}(a_2),\alpha_{1}(a_4)\big)\big)\otimes[\beta_{2}^{2}(x_3), \beta_{2}^{2}(x_5), [\beta_{2}(x_1), \beta_{2}(x_2), \alpha_{2}(x_4)]]\\
&&+(-1)^{(|x_{3}|+|x_{4}|)(|x_{1}|+|x_{2}|)}\times\\&&\mu\big(\beta_{1}^{2}(a_3),\beta_{1}^{2}(a_4),\mu\big(\beta_{1}(a_1)\beta_{1}(a_2)\alpha_{1}(a_5)\big)\otimes[\beta_{2}^{2}(x_3)), \beta_{2}^{2}(x_4), [\beta_{2}(x_1), \beta_{2}(x_2), \alpha_{2}(x_5)]]\\
&=&\mu\big(\beta_{1}^{2}(a_1),\beta_{1}^{2}(a_2),\mu\big(\beta_{1}(a_3),\beta_{1}(a_4),\alpha_{1}(a_5)\big)\big)\otimes\\&&\big((-1)^{(|x_{4}|+|x_{5}|)(|x_{1}|+|x_{2}|+|x_{3}|)}[\beta_{2}^{2}(x_4), \beta_{2}^{2}(x_5), [\beta_{2}(x_1), \beta_{2}(x_2), \alpha_{2}(x_3)]]\\
&&-(-1)^{(|x_{3}|+|x_{5}|)(|x_{1}|+|x_{2}|)+|x_{4}||x_{5}|}[\beta_{2}^{2}(x_3), \beta_{2}^{2}(x_5), [\beta_{2}(x_1), \beta_{2}(x_2), \alpha_{2}(x_4)]]\\&&(-1)^{(|x_{3}|+|x_{4}|)(|x_{1}|+|x_{2}|)}[\beta_{2}^{2}(x_3), \beta_{2}^{2}(x_4), [\beta_{2}(x_1), \beta_{2}(x_2), \alpha_{2}(x_5)]]\big)\\
&=&\mu\big(\beta_{1}^{2}(a_1),\beta_{1}^{2}(a_2),\mu\big(\beta_{1}(a_3)\beta_{1}(a_4)\alpha_{1}(a_5)\big)\big)\otimes[\beta_{2}^{2}(x_1), \beta_{2}^{2}(x_2), [\beta_{2}(x_3), \beta_{2}(x_4)\big), \alpha_{2}(x_5)]]\\
&=&[\beta^2(a_1\otimes x_1), \beta^2(a_2\otimes x_2), [\beta(a_3\otimes x_3), \beta(a_4\otimes x_4), \alpha(a_5\otimes x_5)]_{A\otimes\mathfrak{g}}]_{A\otimes \mathfrak{g}},\end{array}$$
where using equation (\ref{eq}) in the second equality.

Thus $(\mathcal{A}\otimes\mathfrak{g},[.,.,.]_{\mathcal{A}\otimes\mathfrak{g}},\alpha,\beta)$ is a $3$-Bihom-Lie superalgebra.
\end{proof}
\begin{prop}
Given two $3$-Bihom-Lie superalgebras $(\mathfrak{g},[.,.,.],\alpha,\beta)$ and $(\mathfrak{g'},[.,.,.]',\alpha',\beta')$.
Then $(\mathfrak{g}\oplus\mathfrak{g'},[.,.,.]_{\mathfrak{g}\oplus\mathfrak{g'}},\alpha+\alpha',\beta+\beta'),$ is a $3$-Bihom-Lie superalgebras,
where the $3-$linear map $[.,.,.]_{\mathfrak{g}\oplus\mathfrak{g'}}:\wedge^{3}(\mathfrak{g}\oplus\mathfrak{g'})\longrightarrow\mathfrak{g}\oplus\mathfrak{g'}$ is given by
$$[u_{1}+v_{1},u_{2}+v_{2},u_{3}+v_{3}]_{\mathfrak{g}\oplus\mathfrak{g'}}=[u_{1},u_{2},u_{3}]+[v_{1},v_{2},v_{3}]',~\forall u_{i}\in\mathfrak{g},v_{i}\in\mathfrak{g'},~i=1,2,3,$$
and the two linear maps $\alpha+\alpha',\beta+\beta':\mathfrak{g}\oplus\mathfrak{g'}\longrightarrow\mathfrak{g}\oplus\mathfrak{g'},$ are given by $\forall u\in\mathfrak{g},~v\in\mathfrak{g'}$
$$\begin{array}{llll}
(\alpha+\alpha')(u+v)&=&\alpha(u)+\alpha'(v),\\
(\beta+\beta')(u+v)&=&\beta(u)+\beta'(v).
\end{array}$$
\end{prop}
\begin{proof}
For any $u_{i}\in\mathfrak{g}_{\theta_{i}},~v_{i}\in\mathfrak{g'}_{\theta_{i}},~i=1,2,3,4,5$ we have:
$$\begin{array}{llll}(\alpha+\alpha')\circ(\beta+\beta')(u_{1}+v_{1})&=&(\alpha+\alpha')(\beta(u_{1})+\beta'(v_{1}))\\
&=&\alpha\circ\beta(u_{1})+\alpha'\circ\beta'(v_{1})\\
&=&\beta\circ\alpha(u_{1})+\beta'\circ\alpha'(v_{1})\\
&=&(\beta+\beta')(\alpha (u_{1})+\alpha'(v_{1}))\\
&=&(\beta+\beta')\circ(\alpha+\alpha')(u_{1}+v_{1})\end{array}.$$
Next, we consider the $3$-Bihom-super-skewsymmetry,
$$\begin{array}{llllllll}&&[(\beta+\beta')(u_{1}+v_{1}),(\beta+\beta')(u_{2}+v_{2}),(\alpha+\alpha')(u_{3}+v_{3})]_{\mathfrak{g}\oplus\mathfrak{g'}}\\
&=&[\beta(u_{1})+\beta'(v_{1}),\beta(u_{2})+\beta'(v_{2}),\alpha(u_{3})+\alpha'(v_{3})]_{\mathfrak{g}\oplus\mathfrak{g'}}\\
&=&[\beta(u_{1}),\beta(u_{2}),\alpha(u_{3})]+[\beta'(v_{1}),\beta'(v_{2}),\alpha'(v_{3})]'\\
&=&-(-1)^{\theta_{1}\theta_{2}}[\beta(u_{2}),\beta(u_{1}),\alpha(u_{3})]-(-1)^{\theta_{1}\theta_{2}}[\beta'(v_{2}),\beta'(v_{1}),\alpha'(v_{3})]'\\
&=&-(-1)^{\theta_{1}\theta_{2}}[(\beta+\beta')(u_{2}+v_{2}),(\beta+\beta')(u_{1}+v_{1}),(\alpha+\alpha')(u_{3}+v_{3})]_{\mathfrak{g}\oplus\mathfrak{g'}}\\
&=&-(-1)^{\theta_{1}\theta_{2}}[(\beta+\beta')(u_{2}+v_{2}),(\beta+\beta')(u_{1}+v_{1}),(\alpha+\alpha')(u_{3}+v_{3})]_{\mathfrak{g}\oplus\mathfrak{g'}}.\end{array}$$
Similarly, we can get
$$\begin{array}{llllllll}&&[(\beta+\beta')(u_{1}+v_{1}),(\beta+\beta')(u_{2}+v_{2}),(\alpha+\alpha')(u_{3}+v_{3})]_{\mathfrak{g}\oplus\mathfrak{g'}}\\
&=&-(-1)^{\theta_{2}\theta_{3}}[(\beta+\beta')(u_{1}+v_{1}),(\beta+\beta')(u_{3}+v_{3}),(\alpha+\alpha')(u_{2}+v_{2})]_{\mathfrak{g}\oplus\mathfrak{g'}}.\end{array}$$
We prove the $3$-Bihom-super-Jacobi identity,
$$\begin{array}{lllll}&&[(\beta+\beta')^{2}(u_{1}+v_{1}),(\beta+\beta')^{2}(u_{2}+v_{2}),[(\beta+\beta')(u_{3}+v_{3}),
(\beta+\beta')(u_{4}+v_{4}),\\
&&(\alpha+\alpha')(u_{5}+v_{5})]_{\mathfrak{g}\oplus\mathfrak{g'}}]_{\mathfrak{g}\oplus\mathfrak{g'}}\\
&=&[\beta^{2}(u_{1})+\beta^{2}(v_{1},\beta^{2}(u_{2}+\beta'^{2}(v_{2}p),[\beta(u_{3}),\beta(u_{4},\alpha(u_{5})]
+[\beta'(v_{3}),\beta'(v_{4}),\alpha'(v_{5})]']_{\mathfrak{g}\oplus\mathfrak{g'}}\\
&=&[\beta^{2}(u_{1}),\beta^{2}(u_{2}),[\beta(u_{3}),\beta(u_{4}),\alpha(u_{5})]]+[\beta'^{2}(v_{1}),\beta'^{2}(v_{2}),[\beta'(v_{3}),\beta'(v_{4}),\alpha'(v_{4})]']'\\
&=&(-1)^{\theta_{3}\theta_{5}}\circlearrowleft_{u_{4},u_{5},u_{3}}(-1)^{\gamma}[\beta^{2}(u_{4}),\beta^{2}(u_{5}),[\beta(u_{1}),\beta(u_{2}),\alpha(u_{3})]]\\
&&+(-1)^{\theta_{3}\theta_{5}}\circlearrowleft_{v_{4},v_{5},v_{3}}(-1)^{\gamma}[\beta^{2}(v_{4}),\beta^{2}(v_{5}),[\beta(v_{1}),\beta(v_{2}),\alpha(v_{3})]']'\\
&=&(-1)^{\theta_{3}\theta_{5}}\circlearrowleft_{(u_{4},v_{4}),(u_{5},v_{5}),(u_{3},v_{3})}(-1)^{\gamma}
[(\beta+\beta')^{2}(u_{4}+v_{4}),(\beta+\beta')^{2}(u_{5}+v_{5}),\\
&&[(\beta+\beta')(u_{1}+v_{1}),
(\beta+\beta')(u_{2}+v_{2}),
(\alpha+\alpha')(u_{3}+v_{3})]_{\mathfrak{g}\oplus\mathfrak{g'}}]_{\mathfrak{g}\oplus\mathfrak{g'}}
\end{array}$$
where $\gamma=(\theta_{4}+\theta_{5})(\theta_{1}+\theta_{2})+\theta_{3}\theta_{4}$.

Then $(\mathfrak{g}\oplus\mathfrak{g'},[.,.,.]_{\mathfrak{g}\oplus\mathfrak{g'}},\alpha+\alpha,\beta+\beta'),$ is a $3$-Bihom-Lie superalgebras
\end{proof}
\begin{defn}
Let $(\mathfrak{g},[.,.,.],\alpha,\beta)$ and $(\mathfrak{g}',[.,.,.]',\alpha',\beta')$ be two $3$-Bihom-Lie superalgebra. A homomorphism $f:\mathfrak{g}\longrightarrow \mathfrak{g}'$ is said to be morphism of $3$-Bihom-Lie superalgebra if
$$\begin{array}{llllllll}
f([x,y,z])=[f(x),f(y),f(z)]',~\forall x,y,z\in\mathfrak{g},\\
f\circ\alpha=\alpha'\circ f\\
f\circ\beta=\beta'\circ f.
\end{array}$$
\end{defn}
Denote by $\phi_{f}=\{x+f(x);~x\in\mathfrak{g}\}\subset\mathfrak{g}\oplus\mathfrak{g}'$ which is the graph of a linear map $f:\mathfrak{g}\longrightarrow\mathfrak{g'}.$
\begin{prop}
A homomorphism $f:(\mathfrak{g},[.,.,.],\alpha,\beta)\longrightarrow(\mathfrak{g}',[.,.,.]',\alpha',\beta')$ is a morphism of $3$-Bihom-Lie superalgebras if and only if the graph
$\mathfrak{g}\oplus\mathfrak{g'}$ is a $3$-Bihom-subalgebra of $(\mathfrak{g}\oplus\mathfrak{g'},[.,.,.]_{\mathfrak{g}\oplus\mathfrak{g'}},\alpha+\alpha',\beta+\beta')$.
\end{prop}
\begin{proof}
Let $f:(\mathfrak{g},[.,.,.],\alpha,\beta)\longrightarrow(\mathfrak{g}',[.,.,.]',\alpha',\beta')$ is a morphism of $3$-Bihom-Lie superalgebras, for any $x,y,z\in\mathfrak{g}$, we have
$$[x+f(x),y+f(y),z+f(z)]_{\mathfrak{g}\oplus\mathfrak{g}'}=[x,y,z]+[f(x),f(y),f(z)]'=[x,y,z]+f([x,y,z]).$$
Thus the graph $\phi_{f}$ is closed under the bracket operation $[.,.,.]_{\mathfrak{g}\oplus\mathfrak{g'}}$. Furthermore, we have
$$(\alpha+\alpha')(x+f(x))=\alpha(x)+\alpha'\circ f(x)=\alpha(x)+f\circ\alpha(x),$$
which implies that $$(\alpha+\alpha')(\phi_{f})\subset\phi_{f}.$$
Similarly,$$(\beta+\beta')(\phi_{f})\subset\phi_{f}.$$
Thus $\phi_{f}$ is a $3$-Bihom-Lie subalgebras of $(\mathfrak{g}\oplus\mathfrak{g'},[.,.,.]_{\mathfrak{g}\oplus\mathfrak{g'}},\alpha+\alpha',\beta+\beta')$.

Conversely, if the graph $\phi_{f}$ is a $3$-Bihom-Lie subalgebras of $(\mathfrak{g}\oplus\mathfrak{g'},[.,.,.]_{\mathfrak{g}\oplus\mathfrak{g'}},\alpha+\alpha',\beta+\beta')$, we have
$$[x+f(x),y+f(y),z+f(z)]_{\mathfrak{g}\oplus\mathfrak{g'}}=[x,y,z]+[f(x),f(y),f(z)]'\in\phi_{f}$$
which implies that $$[f(x),f(y),f(z)]'=f([x,y,z]).$$
Furthermore, $(\alpha+\alpha')(\phi_{f})\subset\phi_{f}$ yields that
$$(\alpha+\alpha')(x+f(x))=\alpha(x)+\alpha'\circ f(x)\in\phi_{f},$$
which is equivalent to the condition $\alpha'\circ f(x)=f\circ \alpha(x)$, i.e, $\alpha'\circ f=f\circ\alpha.$ Similarly, $\beta'\circ f=f\circ\beta.$ Therefore, $f$ is a morphism of
$3$-Bihom-Lie superalgebras.

\end{proof}
\begin{defns}
\begin{enumerate}
\item
A $3$-Bihom-Lie superalgebra $(\mathfrak{g},[.,.,.],\alpha,\beta)$ is says regular $3$-Bihom-Lie superalgebra, if $\alpha^{-s}\beta^{-r}$ is the inverse of $\alpha^{s}\beta^{r}.$
\item
Let $(\mathfrak{g},[.,.,.],\alpha,\beta)$ be a $3-$Bihom-Lie superalgebra. A linear map $\mathfrak{D}:\mathfrak{g}\longrightarrow\mathfrak{g}$ is called a derivation if it satisfies for all $x,y,z\in\mathfrak{g}$,
$$\mathfrak{D}\circ\alpha=\alpha\circ\mathfrak{D},~~\mathfrak{D}\circ\beta=\beta\circ\mathfrak{D},$$
$$\mathfrak{D}([x,y,z])=[\mathfrak{D}(x),y,z]+(-1)^{|x||\mathfrak{D}|}[x,\mathfrak{D}(y),z]+(-1)^{(|x|+|y|)|\mathfrak{D}|}[x,y,\mathfrak{D}(z)],$$
and it is called an $(\alpha^{s}\beta^{r})$-derivation of $(\mathfrak{g},[.,.,.],\alpha,\beta)$, if satisfies:
$$\mathfrak{D}\circ\alpha=\alpha\circ\mathfrak{D},~~\mathfrak{D}\circ\beta=\beta\circ\mathfrak{D},$$
$$\begin{array}{llllll}\mathfrak{D}([x,y,z])&=&[\mathfrak{D}(x),\alpha^{s}\beta^{r}(y),\alpha^{s}\beta^{r}(z)]+(-1)^{|x||\mathfrak{D}|}[\alpha^{s}\beta^{r}(x),\mathfrak{D}(y),\alpha^{s}\beta^{r}(z)]\\
&+&(-1)^{(|x|+|y|)|\mathfrak{D}|}[\alpha^{s}\beta^{r}(x),\alpha^{s}\beta^{r}(y),\mathfrak{D}(z)].\end{array}$$
denote by $Der_{\alpha^{s}\beta^{r}}(\mathfrak{g})$ the set of $\alpha^{s}\beta^{r}$-derivations of $(\mathfrak{g},[.,.,.],\alpha,\beta)$ and set
$$Der(\mathfrak{g}):=\bigoplus_{s,r\geq 0}Der_{\alpha^{s}\beta^{r}}(\mathfrak{g}).$$
\end{enumerate}
\end{defns}
We show that $Der(\mathfrak{g})$ is equipped with a Lie superalgebra structure defined by $[\mathfrak{D},\mathfrak{D'}]=\mathfrak{D}\mathfrak{D'}-(-1)^{|\mathfrak{D}||\mathfrak{D'}|}\mathfrak{D'}\mathfrak{D}$. Note that for any $\mathfrak{D}\in Der_{\alpha^{s}\beta^{r}}(\mathfrak{g})$ and $\mathfrak{D}'\in Der_{\alpha^{s'}\beta^{r'}}(\mathfrak{g}')$,
we have $[\mathfrak{D},\mathfrak{D'}]\in Der_{(\alpha^{s+s'}\beta^{r+r'})}(\mathfrak{g})$

Now let $(\mathfrak{g},[.,.,.],\alpha,\beta)$ be a regular $3-$Bihom-Lie superalgebra, for any $u_{1},u_{2}\in\mathfrak{g}$ satisfying $\alpha(u_{1})=\beta(u_{1})=u_{1},~\alpha(u_{2})=\beta(u_{2})=u_{2}$, define $\mathfrak{D}_{s,r}(u_{1},u_{2})\in End(\mathfrak{g})$ by
$$\mathfrak{D}_{r,s}(u_{1},u_{2})(w)=[u_{1},u_{2},\alpha^{r}\beta^{s}(w)],~\forall w\in\mathfrak{g},$$
where $|\mathfrak{D}_{r,s}(u_{1},u_{2})|=|u_{1}|+|u_{2}|$.
\begin{prop}Let $(\mathfrak{g},[.,.,.],\alpha,\beta)$ be a regular $3-$Bihom-Lie superalgebra. Then
$\mathfrak{D}_{r,s}(u_{1},u_{2})$ is an $\alpha^{r}\beta^{s+1}$-derivation. We call an inner $\alpha^{r}\beta^{s+1}$-derivation.
\end{prop}
\begin{proof}
For any $w\in\mathfrak{g}$,
$$\begin{array}{llllll}\mathfrak{D}_{r,s}(u_{1},u_{2})(\alpha(w))&=&[u_{1},u_{2},\alpha^{r}\beta^{s}\alpha(w)]\\&=&[\alpha(u_{1}),\alpha(u_{2}),\alpha^{r}\beta^{s}\alpha(w)]\\&=&
\alpha([u_{1},u_{2},\alpha^{r}\beta^{s}(w)]\\&=&\alpha\circ\mathfrak{D}_{r,s}(u_{1},u_{2})(w).\end{array}$$
Similarly, $\mathfrak{D}_{r,s}(u_{1},u_{2})(\beta(w))=\beta\circ\mathfrak{D}_{r,s}(u_{1},u_{2})(w).$
$$\begin{array}{lllllll}&&\mathfrak{D}_{r,s}(u_{1},u_{2})([u,v,w])\\
&=&[u_{1},u_{2},\alpha^{r}\beta^{s}[u,v,w]]\\
&=&[\beta^{2}(u_{1}),\beta^{2}(u_{2}),[\beta\alpha^{r}\beta^{s-1}(u),\beta\alpha^{r}\beta^{s-1}(v),\alpha\alpha^{r-1}\beta^{s}(w)]]\\
&=&(-1)^{(|v|+|w|)(|u_{1}|+|u_{2}|+|u|)}[\beta^{2}\alpha^{r}\beta^{s-1}(v),\beta^{2}\alpha^{r-1}\beta^{s}(w),[\beta(u_{1}),\beta(u_{2}),\alpha\alpha^{r}\beta^{s-1}(u)]]\\
&&-(-1)^{(|u|+|w|)(|u_{1}|+|u_{2}|)+|v||w|}[\beta^{2}\alpha^{r}\beta^{s-1}(u),\beta^{2}\alpha^{r-1}\beta^{s}(w),[\beta(u_{1}),\beta(u_{2}),\alpha\alpha^{r}\beta^{s-1}(v)]]\\
&&+(-1)^{(|u|+|v|)(|u_{1}|+|u_{2}|)}[\beta^{2}\alpha^{r}\beta^{s-1}(u),\beta^{2}\alpha^{r}\beta^{s-1}(v),[\beta(u_{1}),\beta(u_{2}),\alpha\alpha^{r-1}\beta^{s}(w)]]\\
&=&(-1)^{(|v|+|w|)(|u_{1}|+|u_{2}|+|u|)}[\beta^{2}\alpha^{r}\beta^{s-1}(v),\beta^{2}\alpha^{r-1}\beta^{s}(w),[\alpha(u_{1}),\alpha(u_{2}),\alpha\alpha^{r}\beta^{s-1}(u)]]\\
&&-(-1)^{(|u|+|w|)(|u_{1}|+|u_{2}|)+|v||w|}[\beta^{2}\alpha^{r}\beta^{s-1}(u),\beta^{2}\alpha^{r-1}\beta^{s}(w),[\alpha(u_{1}),\alpha(u_{2}),\alpha\alpha^{r}\beta^{s-1}(v)]]\\
&&+(-1)^{(|u|+|v|)(|u_{1}|+|u_{2}|)}[\beta^{2}\alpha^{r}\beta^{s-1}(u),\beta^{2}\alpha^{r}\beta^{s-1}(v),[\alpha(u_{1}),\alpha(u_{2}),\alpha\alpha^{r-1}\beta^{s}(w)]]\\
&=&(-1)^{(|v|+|w|)(|u_{1}|+|u_{2}|+|u|)}[\beta\alpha^{r}\beta^{s}(v),\beta\alpha^{r-1}\beta^{s+1}(w),\alpha([u_{1},u_{2},\alpha^{r}\beta^{s-1}(u)])]\\
&&-(-1)^{(|u|+|w|)(|u_{1}|+|u_{2}|)+|v||w|}[\beta\alpha^{r}\beta^{s}(u),\beta\alpha^{r-1}\beta^{s+1}(w),\alpha([u_{1},u_{2},\alpha^{r}\beta^{s-1}(v)])]\\
&&+(-1)^{(|u|+|v|)(|u_{1}|+|u_{2}|)}[\beta\alpha^{r}\beta^{s}(u),\beta\alpha^{r}\beta^{s}(v),\alpha([u_{1},u_{2},\alpha^{r-1}\beta^{s}(w)])]\\
&=&(-1)^{(|v|+|w|)(|u_{1}|+|u_{2}|+|u|)}(-1)^{(|u_{1}|+|u_{2}|)(|v|+|w|+|u|)}[\beta([u_{1},u_{2},\alpha^{r}\beta^{s-1}(u)]),\beta\alpha^{r}\beta^{s}(v),\alpha\alpha^{r-1}\beta^{s+1}(w)]\\
&&-(-1)^{(|u|+|w|)(|u_{1}|+|u_{2}|)+|v||w|}(-1)^{(|u_{1}|+|u_{2}|+|v|)|w|}[\beta\alpha^{r}\beta^{s}(u),\beta([u_{1},u_{2},\alpha^{r}\beta^{s-1}(v)]),\alpha\alpha^{r-1}\beta^{s+1}(w)]\\
&&+(-1)^{(|u|+|v|)(|u_{1}|+|u_{2}|)}[\beta\alpha^{r}\beta^{s}(u),\beta\alpha^{r}\beta^{s}(v),\alpha([u_{1},u_{2},\alpha^{r-1}\beta^{s}(w)])]\\
&=&[\mathfrak{D}_{r,s}(u_{1},u_{2})(u),\alpha^{r}\beta^{s+1}(v),\alpha\alpha^{r-1}\beta^{s+1}(w)]\\
&&-(-1)^{|u|(|u_{1}|+|u_{2}|)}[\beta\alpha^{r}\beta^{s}(u),\mathfrak{D}_{r,s}(u_{1},u_{2})(v)),\alpha\alpha^{r-1}\beta^{s+1}(w)]
\end{array}$$
$$\begin{array}{lllllll}
&&+(-1)^{(|u|+|v|)(|u_{1}|+|u_{2}|)}[\alpha^{r}\beta^{s+1}(u),\alpha^{r}\beta^{s+1}(v),\mathfrak{D}_{r,s}(u_{1},u_{2})(w)]\\
&=&[\mathfrak{D}_{r,s}(u_{1},u_{2})(u),\alpha^{r}\beta^{s+1}(v),\alpha^{r}\beta^{s+1}(w)]\\
&&-(-1)^{|u||\mathfrak{D}_{r,s}(u_{1},u_{2})|}[\alpha^{r}\beta^{s+1}(u),\mathfrak{D}_{r,s}(u_{1},u_{2})(v)),\alpha^{r}\beta^{s+1}(w)]\\
&&+(-1)^{(|u|+|v|)|\mathfrak{D}_{r,s}(u_{1},u_{2})|}[\alpha^{r}\beta^{s+1}(u),\alpha^{r}\beta^{s+1}(v),\mathfrak{D}_{r,s}(u_{1},u_{2})(w)].
\end{array}$$
Therefore, $\mathfrak{D}_{r,s}(u_{1},u_{2})$ is an $\alpha^{r}\beta^{s+1}$-derivation.
\end{proof}
\section{REPRESENTATIONS AND $T_{\theta}$-EXTENSIONS OF $3$-BIHOM-LIE SUPERALGEBRAS}
\begin{defn}
Let $(\mathfrak{g},[.,.,.],\alpha,\beta)$ be a $3$-Bihom-Lie superalgebra. A representation of $\mathfrak{g}$ is a $4-$tuple
$(M,\rho,\alpha_{M},\beta_{M}),$ where $M$ is a vector superspace, $\alpha_{M},~\beta_{M}\in End(M)$ are two commuting linear maps and $\rho:\mathfrak{g}\times\mathfrak{g}\longrightarrow End(M)$ is a super-skewsymmetry bilinear map, such that for all $u,v,w,x,y,z\in\mathfrak{g},$
\begin{enumerate}
\item
$~\rho(\alpha(u),\alpha(v))\circ \alpha_{M}=\alpha_{M}\circ\rho(u,v),$\\
\item
$~\rho(\beta(u),\beta(v))\circ\beta_{M}=\beta_{M}\circ\rho(u,v).$\\
\item
$~\rho(\alpha\beta(u),\alpha\beta(v))\circ\rho(x,y)=(-1)^{(|u|+|v|)(|x|+|y|)}\rho(\beta(x),\beta(y))\circ\rho(\alpha(u),\alpha(v))$\\
$+\rho([\beta(u),\beta(v),x],\beta(y))\circ\beta_{M}+(-1)^{|x|(|u|+|v|)}\rho(\beta(x),[\beta(u),\beta(v),y])\circ\beta_{M},$\\
\item
$\rho([\beta(u),\beta(v),x],\beta(y))\circ\beta_{M}=(-1)^{|u|(|x|+|v|)}\rho(\alpha\beta(v),\beta(x))\circ\rho(\alpha(u),y)$\\
$+(-1)^{|x|(|u|+|v|)}\rho(\beta(x),\alpha\beta(u))\circ\rho(\alpha(v),y)+\rho((\alpha\beta(u),\alpha\beta(v))\circ\rho(x,y).$
\end{enumerate}
\end{defn}
\begin{prop}\label{prop11}
Let $(\mathfrak{g},[.,.,.],\alpha,\beta)$ be a $3$-Bihom-Lie superalgebra and $(M,\rho,\alpha_{M},\beta_{M})$ a representation of $\mathfrak{g}$. Assume that the maps $\alpha$ and $\beta_{M}$ are surjective. Then $\mathfrak{g}\ltimes M:=(\mathfrak{g}\oplus M,[.,.,.]_{\rho},\alpha+\alpha_{M},\beta+\beta_{M})$ is a $3-$Bihom-Lie superalgebra, where $\alpha+\alpha_{M},~\beta+\beta_{M}:\mathfrak{g}\oplus M\longrightarrow\mathfrak{g}\oplus M$ are defined by $(\alpha+\alpha_{M})(u+x)=\alpha(u)+\alpha_{M}(x)$ and $(\beta+\beta_{M})(u+x)=\beta(u)+\beta_{M}(x),$ and the bracket
$[.,.,.]_{\rho}$ is defined by
$$\begin{array}{llll}&&[u+x,v+y,w+z]_{\rho}=[u,v,w]+\rho(u,v)(z)-(-1)^{|v||w|}\rho(u,\alpha^{-1}\beta(w))(\alpha_{M}\beta_{M}^{-1}(y))\\
&+&(-1)^{|u|(|v|+|w|)}\rho(v,\alpha^{-1}\beta(w))(\alpha_{M}\beta_{M}^{-1}(x)),\end{array}$$
for all $u,v,w\in\mathfrak{g}$ and $x,y,z\in M$. We call $\mathfrak{g}\ltimes M$ the semidirect product of the $3$-Bihom-Lie superalgebra $(\mathfrak{g},[.,.,.],\alpha,\beta)$ and $M$.
\end{prop}
\begin{proof}
First we show $(\alpha+\alpha_{M})\circ(\beta+\beta_{M})$ from the fact $\alpha\circ\beta=\beta\circ\alpha,$ and $\alpha_{M}\circ\beta_{M}=\beta_{M}\circ\alpha_{M}$.

$\forall u,v,w\in\mathfrak{g},~x,y,z\in M$
$$\begin{array}{lllllllllll}&&[(\alpha+\alpha_{M})(u+x),(\alpha+\alpha_{M})(v+y),(\alpha+\alpha_{M})(w+z)]_{\rho}\\
&=&[\alpha(u)+\alpha_{M}(x),\alpha(v)+\alpha_{M}(y),\alpha(w)+\alpha_{M}(z)]_{\rho}\\
&=&[\alpha(u),\alpha(v),\alpha(w)]+\rho(\alpha(u),\alpha(v))(\alpha_{M}(z))\\
&&-(-1)^{|v||w|}\rho(\alpha(u),\alpha^{-1}\beta\alpha(w))(\alpha_{M}\beta_{M}^{-1}\alpha_{M}(y))\\
&&+(-1)^{|u|(|v|+|w|)}\rho(\alpha(v),\alpha^{-1}\beta\alpha(w))(\alpha_{M}\beta_{M}^{-1}\alpha_{M}(x))
\end{array}$$
$$\begin{array}{llllll}
&=&\alpha([u,v,w])+\alpha_{M}\rho(u,v)(z)-(-1)^{|v||w|}\alpha_{M}\rho(u,\alpha^{-1}\beta(w))(\alpha_{M}\beta_{M}^{-1}(y))\\
&&+(-1)^{|u|(|v|+|w|)}\alpha_{M}\rho(v,\alpha^{-1}\beta(w))(\alpha_{M}\beta_{M}^{-1}(x))\\
&=&(\alpha+\alpha_{M})([u,v,w]+\rho(u,v)(z)-(-1)^{|v||w|}\rho(u,\alpha^{-1}\beta(w))(\alpha_{M}\beta_{M}^{-1}(y))\\
&&+(-1)^{|u|(|v|+|w|)}\rho(v,\alpha^{-1}\beta(w))(\alpha_{M}\beta_{M}^{-1}(x)))\\
&=&(\alpha+\alpha_{M})([u+x,v+y,w+z]_{\rho}).
\end{array}$$
Similarly,
$$\begin{array}{lllllllllll}&&[(\beta+\beta_{M})(u+x),(\beta+\beta_{M})(v+y),(\beta+\beta_{M})(w+z)]_{\rho}~~~~~~~~~~~~~~~~~~~~~~~~~~~~~~~~\\
&=&(\beta+\beta_{M})([u+x,v+y,w+z]_{\rho}).
\end{array}$$
Next we show that $[.,.,.]_{\rho}$ satisfies $3$-Bihom-super-skewsymmetry. Let $\forall u,v,w\in\mathfrak{g},~x,y,z\in M$,
$$\begin{array}{llll}&&[(\beta+\beta_{M})(u+x),(\beta+\beta_{M})(v+y),(\alpha+\alpha_{M})(w+z)]_{\rho}\\
&=&[\beta(u)+\beta_{M}(x),\beta(v)+\beta_{M}(y),\alpha(w)+\alpha_{M}(z)]_{\rho}\\
&=&[\beta(u),\beta(v),\alpha(w)]+\rho(\beta(u),\beta(v))(\alpha_{M}(z))\\-&&(-1)^{|v||w|}\rho(\beta(u),\alpha^{-1}\beta\alpha(w))(\alpha_{M}\beta_{M}^{-1}\beta_{M}(y))\\
&&+(-1)^{|u|(|v|+|w|)}\rho(\beta(v),\alpha^{-1}\beta\alpha(w))(\alpha_{M}\beta_{M}^{-1}\beta_{M}(x))\\
&=&-(-1)^{|u||v|}[\beta(v),\beta(u),\alpha(w)]-(-1)^{|u||v|}\rho(\beta(v),\beta(u))(\alpha_{M}(z))\\&&+(-1)^{|u|(|v|+|w|)}\rho(\beta(v),\alpha^{-1}\beta\alpha(w))(\alpha_{M}\beta_{M}^{-1}\beta_{M}(x))\\
&&-(-1)^{|v||w|}\rho(\beta(u),\alpha^{-1}\beta\alpha(w))(\alpha_{M}\beta_{M}^{-1}\beta_{M}(y))\\
&=&-(-1)^{|u||v|}([\beta(v),\beta(u),\alpha(w)]+\rho(\beta(v),\beta(u))(\alpha_{M}(z))\\&&-(-1)^{|u||w|}\rho(\beta(v),\alpha^{-1}\beta\alpha(w))(\alpha_{M}\beta_{M}^{-1}\beta_{M}(x))\\
&&+(-1)^{|v|(|u|+|w|)}\rho(\beta(u),\alpha^{-1}\beta\alpha(w))(\alpha_{M}\beta_{M}^{-1}\beta_{M}(y)))\\
&=&-(-1)^{|u||v|}[(\beta+\beta_{M})(v+y),(\beta+\beta_{M})(u+x),(\alpha+\alpha_{M})(w+z)]_{\rho}
\end{array}$$
\end{proof}
Similarly,
$$\begin{array}{llll}&&[(\beta+\beta_{M})(u+x),(\beta+\beta_{M})(v+y),(\alpha+\alpha_{M})(w+z)]_{\rho}\\
&=&-(-1)^{|v||w|}[(\beta+\beta_{M})(u+x),(\beta+\beta_{M})(w+z),(\alpha+\alpha_{M})(v+y)]_{\rho}
\end{array}$$
Finally, we show that $[.,.,.]_{\rho}$ satisfies $3$-Bihom-super-Jacobi. Let $\forall u_{i}\in\mathfrak{g},~x_{i}\in M,~i=1,2,3,4,5$,
$$\begin{array}{lllll}
&&[(\beta+\beta_{M})^{2}(u_{1}+x_{1}),(\beta+\beta_{M})^{2}(u_{2}+x_{2}),[(\beta+\beta_{M})(u_{3}+x_{3}),\\&&(\beta+\beta_{M})(u_{4}+x_{4}),(\alpha+\alpha_{M})(u_{5}+x_{5})]_{\rho}]_{\rho}\\
&=&[\beta^{2}(u_{1})+\beta_{M}^{2}(x_{1}),\beta^{2}(u_{2})+\beta_{M}^{2}(x_{2}),[\beta(u_{3})+\beta_{M}(x_{3}),\beta(u_{4})+\beta_{M}(x_{4}),\alpha(u_{5})+\alpha_{M}(x_{5})]_{\rho}]_{\rho}\\
&=&[\beta^{2}(u_{1})+\beta_{M}^{2}(x_{1}),\beta^{2}(u_{2})+\beta_{M}^{2}(x_{2}),[\beta(u_{3}),\beta(u_{4}),\alpha(u_{5})]+\rho(\beta(u_{3}),\beta(u_{4}))(\alpha_{M}(x_{5}))\\
&&-(-1)^{|u_{4}||u_{5}|}\rho(\beta(u_{3}),\beta(u_{5}))(\alpha_{M}(x_{4}))+(-1)^{|u_{3}|(|u_{4}|+|u_{5}|)}\rho(\beta(u_{4}),\beta(u_{5}))(\alpha_{M}(x_{3}))]_{\rho}\\
&=&[\beta^{2}(u_{1}),\beta^{2}(u_{2}),[\beta(u_{3}),\beta(u_{4}),\alpha(u_{5})]]+\rho(\beta^{2}(u_{1}),\beta^{2}(u_{2}))(\rho(\beta(u_{3}),\beta(u_{4}))(\alpha_{M}(x_{5}))\\
&&-(-1)^{|u_{4}||u_{5}|}\rho(\beta(u_{3}),\beta(u_{5}))(\alpha_{M}(x_{4}))+(-1)^{|u_{3}|(|u_{4}|+|u_{5}|)}\rho(\beta(u_{4}),\beta(u_{5}))(\alpha_{M}(x_{3})))\\
&&-(-1)^{|u_{2}|(|u_{3}|+|u_{4}|+|u_{5}|)}\rho(\beta^{2}(u_{1}),\alpha^{-1}\beta([\beta(u_{3}),\beta(u_{4}),\alpha(u_{5})]))(\alpha_{M}\beta_{M}(x_{2}))\\
&&+(-1)^{|u_{1}|(|u_{2}|+|u_{3}|+|u_{4}|+|u_{5}|)}\rho(\beta^{2}(u_{2}),\alpha^{-1}\beta([\beta(u_{3}),\beta(u_{4}),\alpha(u_{5})]))(\alpha_{M}\beta_{M}(x_{1}))\\
&=&(-1)^{(|u_{4}|+|u_{5}|)(|u_{1}|+|u_{2}|+|u_{3}|)}[\beta^{2}(u_{4}),\beta^{2}(u_{5}),[\beta(u_{1}),\beta(u_{2}),\alpha(u_{3})]]\\
&&-(-1)^{(|u_{3}|+|u_{5}|)(|u_{1}|+|u_{2}|)+|u_{4}||u_{5}|)}[\beta^{2}(u_{3}),\beta^{2}(u_{5}),[\beta(u_{1}),\beta(u_{2}),\alpha(u_{4})]]\\
&&+(-1)^{(|u_{3}|+|u_{4}|)(|u_{1}|+|u_{2}|)}[\beta^{2}(u_{3}),\beta^{2}(u_{4}),[\beta(u_{1}),\beta(u_{2}),\alpha(u_{5})]]\\
&&+(-1)^{(|u_{1}|+|u_{2}|)(|u_{3}|+|u_{4}|)}\rho(\beta^{2}(u_{3}),\beta^{2}(u_{4}))\rho(\beta(u_{1}),\beta(u_{2}))(\alpha_{M}(x_{5}))\\
&&+\rho(\alpha^{-1}\beta([\beta(u_{1}),\beta(u_{2}),\alpha(u_{3})]),\beta^{2}(u_{4}))(\alpha_{M})\beta_{M})(x_{5}))\\
&&+(-1)^{(|u_{3}|)(|u_{1}|+|u_{2}|)}\rho(\beta^{2}(u_{3}),\alpha^{-1}\beta([\beta(u_{1}),\beta(u_{2}),\alpha(u_{4})]))(\alpha_{M}\beta_{M}(x_{5}))\\
&&-(-1)^{(|u_{1}|+|u_{2}|)(|u_{3}|+|u_{5}|)+|u_{4}||u_{5}|}\rho(\beta^{2}(u_{3}),\beta^{2}(u_{5}))\rho(\beta(u_{1}),\beta(u_{2}))(\alpha_{M}(x_{4}))\\
&&-(-1)^{|u_{4}||u_{5}|}\rho(\alpha^{-1}\beta([\beta(u_{1}),\beta(u_{2}),\alpha(u_{3})]),\beta^{2}(u_{5}))(\alpha_{M})\beta_{M})(x_{4}))\\
&&-(-1)^{(|u_{3}|)(|u_{1}|+|u_{2}|)+|u_{4}||u_{5}|}\rho(\beta^{2}(u_{3}),\alpha^{-1}\beta([\beta(u_{1}),\beta(u_{2}),\alpha(u_{5})]))(\alpha_{M}\beta_{M}(x_{4}))\\
&&+(-1)^{(|u_{1}|+|u_{2}|)(|u_{3}|+|u_{4}|)+|u_{3}|(|u_{4}|+|u_{5}|)}\rho(\beta^{2}(u_{4}),\beta^{2}(u_{5}))\rho(\beta(u_{1}),\beta(u_{2}))(\alpha_{M}(x_{3}))\\
&&+(-1)^{|u_{3}|(|u_{4}|+|u_{5}|)}\rho(\alpha^{-1}\beta([\beta(u_{1}),\beta(u_{2}),\alpha(u_{4})]),\beta^{2}(u_{5}))(\alpha_{M})\beta_{M})(x_{5}))\\
&&+(-1)^{(|u_{3}|)(|u_{1}|+|u_{2}|)+|u_{3}|(|u_{4}|+|u_{5}|)}\rho(\beta^{2}(u_{4}),\alpha^{-1}\beta([\beta(u_{1}),\beta(u_{2}),\alpha(u_{5})]))(\alpha_{M}\beta_{M}(x_{3}))\\
&&+(-1)^{(|u_{1}|+|u_{2}|)(|u_{3}|+|u_{4}|+|u_{5}|)}\rho(\alpha^{-1}\beta([\beta(u_{3}),\beta(u_{4}),\alpha(u_{5})]),\beta^{2}(u_{1}))(\alpha_{M}\beta_{M}(x_{2}))\\
&&-(-1)^{(|u_{1}|+|u_{2}|)(|u_{3}|+|u_{4}|+|u_{5}|)+|u_{1}||u_{2}|}\rho(\alpha^{-1}\beta([\beta(u_{3}),\beta(u_{4}),\alpha(u_{5})]),\beta^{2}(u_{2}))(\alpha_{M}\beta_{M}(x_{1}))\\
&=&(-1)^{(|u_{4}|+|u_{5}|)(|u_{1}|+|u_{2}|+|u_{3}|)}[\beta^{2}(u_{4}),\beta^{2}(u_{5}),[\beta(u_{1}),\beta(u_{2}),\alpha(u_{3})]]
\end{array}$$
$$\begin{array}{lllll}
&&-(-1)^{(|u_{3}|+|u_{5}|)(|u_{1}|+|u_{2}|)+|u_{4}||u_{5}|)}[\beta^{2}(u_{3}),\beta^{2}(u_{5}),[\beta(u_{1}),\beta(u_{2}),\alpha(u_{4})]]\\
&&+(-1)^{(|u_{3}|+|u_{4}|)(|u_{1}|+|u_{2}|)}[\beta^{2}(u_{3}),\beta^{2}(u_{4}),[\beta(u_{1}),\beta(u_{2}),\alpha(u_{5})]]\\
&&+(-1)^{(|u_{1}|+|u_{2}|)(|u_{3}|+|u_{4}|)}\rho(\beta^{2}(u_{3}),\beta^{2}(u_{4}))\rho(\beta(u_{1}),\beta(u_{2}))(\alpha_{M}(x_{5}))\\
&&+\rho(\alpha^{-1}\beta([\beta(u_{1}),\beta(u_{2}),\alpha(u_{3})]),\beta^{2}(u_{4}))(\alpha_{M})\beta_{M})(x_{5}))\\
&&+(-1)^{(|u_{3}|)(|u_{1}|+|u_{2}|)}\rho(\beta^{2}(u_{3}),\alpha^{-1}\beta([\beta(u_{1}),\beta(u_{2}),\alpha(u_{4})]))(\alpha_{M}\beta_{M}(x_{5}))\\
&&-(-1)^{(|u_{1}|+|u_{2}|)(|u_{3}|+|u_{5}|)+|u_{4}||u_{5}|}\rho(\beta^{2}(u_{3}),\beta^{2}(u_{5}))\rho(\beta(u_{1}),\beta(u_{2}))(\alpha_{M}(x_{4}))\\
&&-(-1)^{|u_{4}||u_{5}|}\rho(\alpha^{-1}\beta([\beta(u_{1}),\beta(u_{2}),\alpha(u_{3})]),\beta^{2}(u_{5}))(\alpha_{M})\beta_{M})(x_{4}))\\
&&-(-1)^{(|u_{3}|)(|u_{1}|+|u_{2}|)+|u_{4}||u_{5}|}\rho(\beta^{2}(u_{3}),\alpha^{-1}\beta([\beta(u_{1}),\beta(u_{2}),\alpha(u_{5})]))(\alpha_{M}\beta_{M}(x_{4}))\\
&&+(-1)^{(|u_{1}|+|u_{2}|)(|u_{3}|+|u_{4}|)+|u_{3}|(|u_{4}|+|u_{5}|)}\rho(\beta^{2}(u_{4}),\beta^{2}(u_{5}))\rho(\beta(u_{1}),\beta(u_{2}))(\alpha_{M}(x_{3}))\\
&&+(-1)^{|u_{3}|(|u_{4}|+|u_{5}|)}\rho(\alpha^{-1}\beta([\beta(u_{1}),\beta(u_{2}),\alpha(u_{4})]),\beta^{2}(u_{5}))(\alpha_{M})\beta_{M})(x_{3}))\\
&&+(-1)^{(|u_{3}|)(|u_{1}|+|u_{2}|)+|u_{3}|(|u_{4}|+|u_{5}|)}\rho(\beta^{2}(u_{4}),\alpha^{-1}\beta([\beta(u_{1}),\beta(u_{2}),\alpha(u_{5})]))(\alpha_{M}\beta_{M}(x_{3}))\\
&&+(-1)^{(|u_{1}|+|u_{2}|+|u_{3}|)(|u_{3}|+|u_{4}|+|u_{5}|)+|u_{3}|}\rho(\beta^{2}(u_{4}),\beta^{2}(u_{5}))\rho(\beta(u_{3}),\beta(u_{1}))(\alpha_{M}(x_{2}))\\
&&+(-1)^{(|u_{1}|+|u_{2}|+|u_{5}|)(|u_{3}|+|u_{4}|+|u_{5}|)+|u_{5}|}\rho(\beta^{2}(u_{5}),\beta^{2}(u_{3}))\rho(\beta(u_{4}),\beta(u_{1}))(\alpha_{M}(x_{2}))\\
&&+(-1)^{(|u_{1}|+|u_{2}|)(|u_{3}|+|u_{4}|+|u_{5}|)}\rho(\beta^{2}(u_{3}),\beta^{2}(u_{4}))\rho(\beta(u_{5}),\beta(u_{1}))(\alpha_{M}(x_{2}))\\
&&-(-1)^{(|u_{1}|+|u_{2}|+|u_{3}|)(|u_{3}|+|u_{4}|+|u_{5}|)+|u_{3}|+|u_{1}||u_{2}|}\rho(\beta^{2}(u_{4}),\beta^{2}(u_{5}))\rho(\beta(u_{3}),\beta(u_{2}))(\alpha_{M}(x_{1}))\\
&&-(-1)^{(|u_{1}|+|u_{2}|)(|u_{4}|+|u_{5}|)+|u_{1}||u_{2}|}\rho(\beta^{2}(u_{5}),\beta^{2}(u_{3}))\rho(\beta(u_{4}),\beta(u_{2}))(\alpha_{M}(x_{1}))\\
&&-(-1)^{(|u_{1}|+|u_{2}|)(|u_{3}|+|u_{4}|+|u_{5}|)}\rho(\beta^{2}(u_{3}),\beta^{2}(u_{4}))\rho(\beta(u_{5}),\beta(u_{2}))(\alpha_{M}(x_{1}))\\
&=&(-1)^{(|u_{4}|+|u_{5}|)(|u_{1}|+|u_{2}|+|u_{3}|)}[\beta^{2}(u_{4}),\beta^{2}(u_{5}),[\beta(u_{1}),\beta(u_{2}),\alpha(u_{3})]]\\
&&+\rho(\beta^{2}(u_{4}),\beta^{2}(u_{5}))((-1)^{(|u_{4}|+|u_{5}|)(|u_{1}|+|u_{2}|+|u_{3}|)}\rho(\beta(u_{1}),\beta(u_{2}))(\alpha_{M}(x_{3}))\\
&&-(-1)^{(|u_{4}|+|u_{5}|)(|u_{1}|+|u_{2}|+|u_{3}|)+|u_{2}||u_{3}|}\rho(\beta(u_{1}),\beta(u_{3}))(\alpha_{M}(x_{2}))\\
&&+(-1)^{(|u_{4}|+|u_{5}|+|u_{1}|)(|u_{1}|+|u_{2}|+|u_{3}|)+|u_{1}|}\rho(\beta(u_{2}),\beta(u_{3}))(\alpha_{M}(x_{1})))\\
&&-(-1)^{|u_{4}|(|u_{1}|+|u_{2}|+|u_{3}|)}\rho(\beta^{2}(u_{4}),\alpha^{-1}\beta([\beta(u_{1}),\beta(u_{2}),\alpha(u_{3})]))(\alpha_{M}\beta_{M}(x_{5}))\\
&&-(-1)^{|u_{5}|(|u_{1}|+|u_{2}|+|u_{3}|+|u_{4}|)}\rho(\beta^{2}(u_{5}),\alpha^{-1}\beta([\beta(u_{1}),\beta(u_{2}),\alpha(u_{3})]))(\alpha_{M}\beta_{M}(x_{4}))\\
&&-(-1)^{(|u_{3}|+|u_{5}|)(|u_{1}|+|u_{2}|)+|u_{4}||u_{5}|}[\beta^{2}(u_{3}),\beta^{2}(u_{5}),[\beta(u_{1}),\beta(u_{2}),\alpha(u_{4})]]\\
&&-\rho(\beta^{2}(u_{3}),\beta^{2}(u_{5}))((-1)^{(|u_{1}|+|u_{2}|)(|u_{3}|+|u_{5}|)+|u_{4}||u_{5}|}(\rho(\beta(u_{1}),\beta(u_{2}))(\alpha_{M}(x_{4}))\\
&&-(-1)^{(|u_{3}|+|u_{5}|)(|u_{1}|+|u_{2}|)+|u_{4}|(|u_{5}|+|u_{2}|)}\rho(\beta(u_{1}),\beta(u_{4}))(\alpha_{M}(x_{2}))\\
&&+(-1)^{(|u_{3}|+|u_{5}|)(|u_{1}|+|u_{2}|)+|u_{4}|(|u_{5}|+|u_{1}|)+|u_{1}||u_{2}|}\rho(\beta(u_{2}),\beta(u_{4}))(\alpha_{M}(x_{1})))\\
&&+(-1)^{|u_{3}|(|u_{1}|+|u_{2}|)}\rho(\beta^{2}(u_{3}),\alpha^{-1}\beta([\beta(u_{1}),\beta(u_{2}),\alpha(u_{4})]))(\alpha_{M}\beta_{M}(x_{5}))\\
&&-(-1)^{|u_{5}|(|u_{1}|+|u_{2}|+|u_{4}|)+|u_{3}|)+|u_{3}||u_{4}|}\rho(\beta^{2}(u_{5}),\alpha^{-1}\beta([\beta(u_{1}),\beta(u_{2}),\alpha(u_{4})]))(\alpha_{M}\beta_{M}(x_{3}))\\
&&+(-1)^{(|u_{3}|+|u_{4}|)(|u_{1}|+|u_{2}|)}[\beta^{2}(u_{3}),\beta^{2}(u_{4}),[\beta(u_{1}),\beta(u_{2}),\alpha(u_{5})]]\\
&&+\rho(\beta^{2}(u_{3}),\beta^{2}(u_{4}))((-1)^{(|u_{1}|+|u_{2}|)(|u_{3}|+|u_{4}|)}\rho(\beta(u_{1}),\beta(u_{2}))(\alpha_{M}(x_{5}))\\
&&-(-1)^{(|u_{1}|+|u_{2}|)(|u_{3}|+|u_{4}|)+|u_{2}||u_{5}|)}\rho(\beta(u_{1}),\beta(u_{5}))(\alpha_{M}(x_{2}))\\
&&+(-1)^{(|u_{1}|+|u_{2}|)(|u_{3}|+|u_{4}|+|u_{5}|)+|u_{2}||u_{5}|}\rho(\beta(u_{2}),\beta(u_{5}))(\alpha_{M}(x_{1}))\\
&&-(-1)^{|u_{3}|(|u_{1}|+|u_{2}|)+|u_{4}||u_{5}|}\rho(\beta^{2}(u_{3}),\alpha^{-1}\beta([\beta(u_{1}),\beta(u_{2}),\alpha(u_{4})]))(\alpha_{M}\beta_{M}(x_{4}))\\
&&+(-1)^{|u_{4}|(|u_{1}|+|u_{2}|)+|u_{3}|(|u_{4}|+|u_{5}|}\rho(\beta^{2}(u_{4}),\alpha^{-1}\beta([\beta(u_{1}),\beta(u_{2}),\alpha(u_{4})]))(\alpha_{M}\beta_{M}(x_{3}))
\end{array}$$
$$\begin{array}{lllll}
&=&(-1)^{(|u_{4}|+|u_{5}|)(|u_{1}|+|u_{2}|+|u_{3}|)}[(\beta+\beta_{M})^{2}(u_{4}+x_{4}),(\beta+\beta_{M})^{2}(u_{5}+x_{5}),[(\beta+\beta_{M})(u_{1}+x_{1}),
\\&&(\beta+\beta_{M})(u_{2}+x_{2}),(\alpha+\alpha_{M})(u_{3}+x_{3})]_{\rho}]_{\rho}\\
&&-(-1)^{(|u_{3}|+|u_{5}|)(|u_{1}|+|u_{2}|)+|u_{4}||u_{5}|}[(\beta+\beta_{M})^{2}(u_{3}+x_{3}),(\beta+\beta_{M})^{2}(u_{5}+x_{5}),[(\beta+\beta_{M})(u_{1}+x_{1}),
\\&&(\beta+\beta_{M})(u_{2}+x_{2}),(\alpha+\alpha_{M})(u_{4}+x_{4})]_{\rho}]_{\rho}\\
&&+(-1)^{(|u_{3}|+|u_{4}|)(|u_{1}|+|u_{2}|)}[(\beta+\beta_{M})^{2}(u_{3}+x_{3}),(\beta+\beta_{M})^{2}(u_{4}+x_{4}),[(\beta+\beta_{M})(u_{1}+x_{1}),
\\&&(\beta+\beta_{M})(u_{2}+x_{2}),(\alpha+\alpha_{M})(u_{5}+x_{5})]_{\rho}]_{\rho}\\
\end{array}$$
 Then $\mathfrak{g}\ltimes M:=(\mathfrak{g}\oplus M,[.,.,.]_{\rho},\alpha+\alpha_{M},\beta+\beta_{M})$ is a $3-$Bihom-Lie superalgebra.
\begin{defn}
Let $(\mathfrak{g},[.,.,.],\alpha,\beta)$ be a $3-$Bihom-Lie superalgebra and $(M,\rho,\alpha_{M},\beta_{M})$ be a representation of $\mathfrak{g}$. If $\theta:
\mathfrak{g}\times\mathfrak{g}\times\mathfrak{g}\rightarrow M$ is a $3$-linear map and satisfies
\begin{enumerate}
\item
$\alpha_{M}\theta(x_{1},x_{2},x_{3})=\theta(\alpha(x_{1}),\alpha(x_{2}),\alpha(x_{3})),$
\item
$\beta_{M}\theta(x_{1},x_{2},x_{3})=\theta(\beta(x_{1}),\beta(x_{2}),\beta(x_{3})),$
\item
$\theta(\beta(x_{1}),\beta(x_{2}),\alpha(x_{3}))=-(-1)^{|x_{1}||x_{2}|}\theta(\beta(x_{2}),\beta(x_{1}),\alpha(x_{3}))=-(-1)^{|x_{2}||x_{3}|}\theta(\beta(x_{1}),\beta(x_{3}),\alpha(x_{2})),$
\item
$\theta(\beta^{2}(x_{1}),\beta^{2}(x_{2}),[\beta(x_{3}),\beta(x_{4}),\alpha(x_{5})])+\rho(\beta^{2}(x_{1}),\beta^{2}(x_{2}))\theta(\beta(x_{3}),\beta(x_{4}),\alpha(x_{5}))$\\
$\begin{array}{lllll}
&=&(-1)^{(|x_{4}|+|x_{5}|)(|x_{1}|+|x_{2}|+|x_{3}|)}\theta(\beta^{2}(x_{4}),\beta^{2}(x_{5}),[\beta(x_{1}),\beta(x_{2}),\alpha(x_{3})])\\
&&+(-1)^{(|x_{4}|+|x_{5}|)(|x_{1}|+|x_{2}|+|x_{3}|)}\rho(\beta^{2}(x_{4}),\beta^{2}(x_{5}))\theta(\beta(x_{1}),\beta(x_{2}),\alpha(x_{3}))\\
&&-(-1)^{(|x_{3}|+|x_{5}|)(|x_{1}|+|x_{2}|)+|x_{4}||x_{5}|}\theta(\beta^{2}(x_{3}),\beta^{2}(x_{5}),[\beta(x_{1}),\beta(x_{2}),\alpha(x_{4})])\\
&&-(-1)^{(|x_{3}|+|x_{5}|)(|x_{1}|+|x_{2}|)+|x_{4}||x_{5}|}\rho(\beta^{2}(x_{3}),\beta^{2}(x_{5}))\theta(\beta(x_{1}),\beta(x_{2}),\alpha(x_{4}))\\
&&+(-1)^{(|x_{3}|+|x_{4}|)(|x_{1}|+|x_{2}|)}\theta(\beta^{2}(x_{3}),\beta^{2}(x_{4}),[\beta(x_{1}),\beta(x_{2}),\alpha(x_{5})])\\
&&+(-1)^{(|x_{3}|+|x_{4}|)(|x_{1}|+|x_{2}|)}\rho(\beta^{2}(x_{3}),\beta^{2}(x_{4}))\theta(\beta(x_{1}),\beta(x_{2}),\alpha(x_{5})),
\end{array}$
\end{enumerate}
where $x_{1},~x_{2},~x_{3},~x_{4},~x_{5}\in\mathfrak{g}$. Then $\theta$ is called a $3$-cocycle associated with $\rho$.
\end{defn}

\begin{prop}
Let $(\mathfrak{g},[.,.,.],\alpha,\beta)$ be a $3$-Bihom-Lie superalgebra and $(M,\rho,\alpha_{M},\beta_{M})$ a representation of $\mathfrak{g}$. Assume that the maps $\alpha$ and $\beta_{M}$ are surjective. If $\theta$ is a $3$-cocycle associated with $\rho$. Then $(\mathfrak{g}\oplus M,[.,.,.]_{\theta},\alpha+\alpha_{M},\beta+\beta_{M})$ is a $3-$Bihom-Lie superalgebra, where $\alpha+\alpha_{M};~\beta+\beta_{M}:\mathfrak{g}\oplus M\longrightarrow\mathfrak{g}\oplus M$ are defined by $(\alpha+\alpha_{M})(u+x)=\alpha(u)+\alpha_{M}(x)$ and $(\beta+\beta_{M})(u+x)=\beta(u)+\beta_{M}(x),$ and the bracket
$[.,.,.]_{\theta}$ is defined by
$$\begin{array}{llll}&&[u+x,v+y,w+z]_{\theta}=[u,v,w]+\theta(u,v,w)+\rho(u,v)(z)-(-1)^{|v||w|}\rho(u,\alpha^{-1}\beta(w))(\alpha_{M}\beta_{M}^{-1}(y))\\
&+&(-1)^{|u|(|v|+|w|)}\rho(v,\alpha^{-1}\beta(w))(\alpha_{M}\beta_{M}^{-1}(x)),\end{array}$$
for all $u,v,w\in\mathfrak{g}$ and $x,y,z\in M$. $(\mathfrak{g}\oplus M,[.,.,.]_{\theta},\alpha+\alpha_{M},\beta+\beta_{M})$ is called the $T_{\theta}$-extension of $(\mathfrak{g},[.,.,.],\alpha,\beta)$ by $M$, denoted by $T_{\theta}(\mathfrak{g}).$
\end{prop}
\begin{proof}
The proof is similar to Proposition \ref{prop11}.
\end{proof}
\begin{prop}\label{propk}
Let $(\mathfrak{g},[.,.,.],\alpha,\beta)$ be a $3-$Bihom-Lie superalgebra and $(M,\rho,\alpha_{M},\beta_{M})$ be a representation of $\mathfrak{g}$.
Assume that the maps $\alpha$ and $\beta$ are surjective. $f:\mathfrak{g}\rightarrow M$ is a linear map such that $f\circ\alpha=\alpha_{M}\circ f$ and
$f\circ\beta=\beta_{M}\circ f$. Then the $3-$linear map $\theta_{f}:\mathfrak{g}\times\mathfrak{g}\times\mathfrak{g}\rightarrow M$ given by
$$\begin{array}{lll}\theta_{f}(x,y,z)&=&f([x,y,z])-\rho(x,y)f(z)+(-1)^{|y||z|}\rho(x,\alpha^{-1}\beta(z))f(\alpha\beta^{-1}(y))\\
&&-(-1)^{(|y|+|z|)|x|}\rho(y,\alpha^{-1}\beta(z))f(\alpha\beta^{-1}(x)),\end{array}$$
for all $x,y,z\in\mathfrak{g}$, is a $3-$cocycle associated with $\rho$.
\end{prop}
\begin{proof}
For any $x,y,z\in\mathfrak{g}$ we have
$$\begin{array}{lllll}&&\theta_{f}(\alpha(x),\alpha(y),\alpha(z))\\
&=&f([\alpha(x),\alpha(y),\alpha(z)])-\rho(\alpha(x),\alpha(y))f(\alpha(z))
+(-1)^{|y||z|}\rho(\alpha(x),\alpha\alpha^{-1}\beta(z))f(\alpha\alpha\beta^{-1}(y))\\
&&-(-1)^{(|y|+|z|)|x|}\rho(\alpha(y),\alpha\alpha^{-1}\beta(z))f(\alpha\alpha\beta^{-1}(x))\\
&=&\alpha_{M}f([x,y,z])-\alpha_{M}\rho(x,y)f(z)
+(-1)^{|y||z|}\alpha_{M}\rho(x,\alpha^{-1}\beta(z))f(\alpha\beta^{-1}(y))\\&&-(-1)^{(|y|+|z|)|x|}\alpha_{M}\rho(y,\alpha^{-1}\beta(z))f(\alpha\beta^{-1}(x))\\
&=&\alpha_{M}(f([x,y,z])-\rho(x,y)f(z)+(-1)^{|y||z|}\rho(x,\alpha^{-1}\beta(z))f(\alpha\beta^{-1}(y))\\
&&-(-1)^{(|y|+|z|)|x|}\rho(y,\alpha^{-1}\beta(z))f(\alpha\beta^{-1}(x)))\\
&=&\alpha_{M}\theta_{f}(x,y,z).
\end{array}$$
We also have $\theta_{f}(\beta(x),\beta(y),\beta(z))=\beta_{M}\theta_{f}(x,y,z).$

Now, we have
$$\begin{array}{lllll}&&\theta_{f}(\beta(x),\beta(y),\alpha(z))\\
&=&f([\beta(x),\beta(y),\alpha(z)])-\rho(\beta(x),\beta(y))(f(\alpha(z))
+(-1)^{|y||z|}\rho(\beta(x),\alpha\alpha^{-1}\beta(z))f(\alpha\beta^{-1}\beta(y))\\
&&-(-1)^{(|y|+|z|)|x|}\rho(\beta(y),\alpha\alpha^{-1}\beta(z))f(\alpha\beta^{-1}\beta(x))\\
&=&-(-1)^{|x||y|}f([\beta(y),\beta(x),\alpha(z)])+(-1)^{|x||y|}(\rho(\beta(y),\beta(x))(f(\alpha(z))\\
&&-(-1)^{(|y|+|z|)|x|}\rho(\beta(x),\alpha\alpha^{-1}\beta(z))f(\alpha\beta^{-1}\beta(y))
+(-1)^{|y||z|}\rho(\beta(y),\alpha\alpha^{-1}\beta(z))f(\alpha\beta^{-1}\beta(x))\\
&=&-(-1)^{|x||y|}\theta_{f}(\beta(y),\beta(x),\alpha(z))
\end{array}$$
Similarly, we can get $\theta_{f}(\beta(x),\beta(y),\alpha(z))=-(-1)^{|y||z|}\theta_{f}(\beta(x),\beta(z),\alpha(y))$

Finally, $\forall x_{1},x_{2},x_{3},x_{4},x_{5}\in\mathfrak{g}$, we have

$\begin{array}{lllll}&&\theta_{f}(\beta^{2}(x_{1}),\beta^{2}(x_{2}),[\beta(x_{3}),\beta(x_{4}),\beta(x_{5})])+\rho(\beta^{2}(x_{1}),\beta^{2}(x_{2}))\theta_{f}(\beta(x_{3}),\beta(x_{4}),\alpha(x_{5}))\\
&=&f([\beta^{2}(x_{1}),\beta^{2}(x_{2}),[\beta(x_{3}),\beta(x_{4}),\alpha(x_{5})]])
-\rho(\beta^{2}(x_{1}),\beta^{2}(x_{2}))f([\beta(x_{3}),\beta(x_{4}),\alpha(x_{5})])\\
&&+(-1)^{|x_{2}|(|x_{3}|+|x_{4}|+|x_{5}|)}\rho(\beta^{2}(x_{1}),\alpha^{-1}\beta([\beta(x_{3}),\beta(x_{4}),\alpha(x_{5})]))f(\alpha\beta(x_{2})\\
&&-(-1)^{|x_{1}|(|x_{2}|+|x_{3}|+|x_{4}|+|x_{5}|)}\rho(\beta^{2}(x_{2}),\alpha^{-1}\beta([\beta(x_{3}),\beta(x_{4}),\alpha(x_{5})]))f(\alpha\beta(x_{1}))\\
&&+\rho(\beta^{2}(x_{1}),\beta^{2}(x_{2}))(f([\beta(x_{3}),\beta(x_{4}),\alpha(x_{5})])-\rho(\beta(x_{3}),\beta(x_{4}))f(\alpha(x_{5})\\
&&+(-1)^{|x_{4}||x_{5}|}\rho(\beta(x_{3}),\beta(x_{5}))f(\alpha(x_{4})-(-1)^{|x_{3}|(|x_{4}|+|x_{5}|)}\rho(\beta(x_{4}),\beta(x_{5}))f(\alpha(x_{3})\\
&=&f([\beta^{2}(x_{1}),\beta^{2}(x_{2}),[\beta(x_{3}),\beta(x_{4}),\alpha(x_{5})]])\\
&&+(-1)^{|x_{2}|(|x_{3}|+|x_{4}|+|x_{5}|)}\rho(\beta^{2}(x_{1}),\alpha^{-1}\beta([\beta(x_{3}),\beta(x_{4}),\alpha(x_{5})]))f(\alpha\beta(x_{2})\\
&&-(-1)^{|x_{1}|(|x_{2}|+|x_{3}|+|x_{4}|+|x_{5}|)}\rho(\beta^{2}(x_{2}),\alpha^{-1}\beta([\beta(x_{3}),\beta(x_{4}),\alpha(x_{5})]))f(\alpha\beta(x_{1})\\
&&+\rho(\beta^{2}(x_{1}),\beta^{2}(x_{2}))(-\rho(\beta(x_{3}),\beta(x_{4}))f(\alpha(x_{5}))+(-1)^{|x_{4}||x_{5}|}\rho(\beta(x_{3}),\beta(x_{5}))f(\alpha(x_{4}))\\
&&-(-1)^{|x_{3}|(|x_{4}|+|x_{5}|)}\rho(\beta(x_{4}),\beta(x_{5}))f(\alpha(x_{3}))\\
&=&(-1)^{(|x_{4}|+|x_{5}|)(|x_{1}|+|x_{2}|+|x_{3}|)}f([\beta^{2}(x_{4}),\beta^{2}(x_{5}),[\beta(x_{1}),\beta(x_{2}),\alpha(x_{3})]])\\
&&-(-1)^{(|x_{1}|+|x_{5}|)(|x_{1}|+|x_{2}|)+|x_{4}||x_{5}|}f([\beta^{2}(x_{3}),\beta^{2}(x_{5}),[\beta(x_{1}),\beta(x_{2}),\alpha(x_{4})]])\\
&&+(-1)^{(|x_{3}|+|x_{4}|)(|x_{1}|+|x_{2}|)}f([\beta^{2}(x_{3}),\beta^{2}(x_{4}),[\beta(x_{1}),\beta(x_{2}),\alpha(x_{5})]])\\&&-(-1)^{(|x_{1}|+|x_{2}|+|x_{3}|)(|x_{3}|+|x_{4}|+|x_{5}|)+|x_{3}|}\rho(\beta^{2}(x_{4}),\beta^{2}(x_{5}))\rho(\beta(x_{3}),\beta(x_{1}))f(\alpha(x_{2}))\\
&&-(-1)^{(|x_{1}|+|x_{2}|+|x_{5}|)(|x_{3}|+|x_{4}|+|x_{5}|)+|x_{5}|}\rho(\beta^{2}(x_{5}),\beta^{2}(x_{3}))\rho(\beta(x_{4}),\beta(x_{1}))f(\alpha(x_{2}))\\
&&-(-1)^{|x_{1}|(|x_{3}|+|x_{4}|+|x_{5}|)}\rho(\beta^{2}(x_{3}),\beta^{2}(x_{4}))\rho(\beta(x_{5}),\beta(x_{1}))f(\alpha(x_{2}))\\
&&+(-1)^{(|x_{1}|+|x_{2}|+|x_{3}|)(|x_{3}|+|x_{4}|+|x_{5}|+|x_{5}|)+|x_{3}|+|x_{1}||x_{2}|}\rho(\beta^{2}(x_{4}),\beta^{2}(x_{5}))\rho(\beta(x_{3}),\beta(x_{2}))f(\alpha(x_{1}))\\
&&+(-1)^{(|x_{1}|+|x_{2}|+|x_{5}|)(|x_{2}|+|x_{3}|+|x_{4}|)+|x_{2}|+|x_{1}||x_{5}|}\rho(\beta^{2}(x_{5}),\beta^{2}(x_{3}))\rho(\beta(x_{4}),\beta(x_{2}))f(\alpha(x_{1}))\\
&&+(-1)^{(|x_{1}|+|x_{2}|)(|x_{3}|+|x_{4}|+|x_{5}|)+|x_{1}||x_{2}|}\rho(\beta^{2}(x_{3}),\beta^{2}(x_{4}))\rho(\beta(x_{5}),\beta(x_{2}))f(\alpha(x_{1}))\\
&&-(-1)^{(|x_{1}|+|x_{2}|)(|x_{3}|+|x_{4}|)}\rho(\beta^{2}(x_{3}),\beta^{2}(x_{4}))\rho(\beta(x_{1}),\beta(x_{2}))f(\alpha(x_{5}))\\
&&-\rho(\alpha^{-1}\beta([\beta(x_{1}),\beta(x_{2}),\alpha(x_{3})]),\beta^{2}(x_{4}))f(\alpha\beta(x_{5})\\
&&-(-1)^{|x_{3}|(|x_{1}|+|x_{2}|)}\rho(\beta^{2}(x_{3}),\alpha^{-1}\beta([\beta(x_{1}),\beta(x_{2}),\alpha(x_{4})]))f(\alpha\beta(x_{5}))\\
&&+(-1)^{(|x_{1}|+|x_{2}|)(|x_{3}|+|x_{5}|)+|x_{4}||x_{5}|}\rho(\beta^{2}(x_{3}),\beta^{2}(x_{5}))\rho(\beta(x_{1}),\beta(x_{2}))f(\alpha(x_{4}))\\
&&+(-1)^{|x_{4}||x_{5}|}\rho(\alpha^{-1}\beta([\beta(x_{1}),\beta(x_{2}),\alpha(x_{3})]),\beta^{2}(x_{5}))f(\alpha\beta(x_{4})\\
&&+(-1)^{(|x_{1}|+|x_{2}|)|x_{3}|+|x_{4}||x_{5}|}\rho(\beta^{2}(x_{3}),\alpha^{-1}\beta([\beta(x_{1}),\beta(x_{2}),\alpha(x_{5})]))f(\alpha\beta(x_{4}))\\
&&-(-1)^{(|x_{1}|+|x_{2}|+|x_{3}|)(|x_{4}||x_{5}|)}\rho(\beta^{2}(x_{4}),\beta^{2}(x_{5}))\rho(\beta(x_{1}),\beta(x_{2}))f(\alpha(x_{3}))\\
&&-(-1)^{|x_{3}|(|x_{4}|+|x_{5}|)}\rho(\alpha^{-1}\beta([\beta(x_{1}),\beta(x_{2}),\alpha(x_{4})]),\beta^{2}(x_{5}))f(\alpha\beta(x_{3})\\
&&-(-1)^{|x_{1}|(|x_{2}|+|x_{4}|)+|x_{2}||x_{4}|}\rho(\beta^{2}(x_{4}),\alpha^{-1}\beta([\beta(x_{2}),\beta(x_{1}),\alpha(x_{5})]))f(\alpha\beta(x_{3}))\\
&=&(-1)^{(|x_{4}|+|x_{5}|)(|x_{1}|+|x_{2}|+|x_{3}|)}f([\beta^2(x_4)\!,\!\beta^2(x_5)\!,\![\beta(x_1)\!,\!\beta(x_2)\!,\!\alpha(x_3)]])\\&&+(-1)^{|x_{4}|(|x_{1}|+|x_{2}|+|x_{3}|)}\rho(\beta^2(x_4),\alpha^{-1}\beta([\beta(x_1),\beta(x_2),\alpha(x_3)]))f\alpha\beta(x_5)\\
&&-(-1)^{|x_{5}|(|x_{1}|+|x_{2}|+|x_{3}|)}\rho(\beta^2(x_5),\alpha^{-1}\beta([\beta(x_1),\beta(x_2),\alpha(x_3)]))f\alpha\beta(x_4)\\
&&+(-1)^{|x_{5}|(|x_{1}|+|x_{2}|+|x_{3}|+|x_{4}|)+|x_{3}||x_{4}|}\rho(\beta^2(x_4),\beta^2(x_5))\big(-\rho(\beta(x_1),\beta(x_2))f\alpha(x_3)\\&&+(-1)^{|x_{2}||x_{3}|}\rho(\beta(x_1),\beta(x_3))f\alpha(x_2)
-(-1)^{|x_{1}|(|x_{2}|+|x_{3}|)}\rho(\beta(x_2),\beta(x_3))f\alpha(x_1)\big)
\end{array}$
$\begin{array}{lllll}
&&-(-1)^{(|x_{3}|+|x_{5}|)(|x_{1}|+|x_{2}|)+|x_{4}||x_{5}|)}f([\beta^2(x_3)\!,\!\beta^2(x_5)\!,\![\beta(x_1)\!,\!\beta(x_2)\!,\!\alpha(x_4)]])\!\\&&-(-1)^{|x_{3}|(|x_{1}|+|x_{2}|)}\!\rho(\beta^2(x_3)\!,\!\alpha^{-1}\beta([\beta(x_1)\!,\!\beta(x_2)\!,\!\alpha(x_4)]))f\alpha\beta(x_5)\\
&&+(-1)^{(|x_{4}|+|x_{5}|)(|x_{1}|+|x_{2}|+|x_{3}|)}\rho(\beta^2(x_5),\alpha^{-1}\beta([\beta(x_1),\beta(x_2),\alpha(x_4)]))f\alpha\beta(x_3)\\
&&-(-1)^{(|x_{3}|+|x_{5}|)(|x_{1}|+|x_{2}|)+|x_{4}||x_{5}|)}\rho(\beta^2(x_3),\beta^2(x_5))\big(-\rho(\beta(x_1),\beta(x_2))f\alpha(x_4)\\&&+\rho(\beta(x_1),\beta(x_4))f\alpha(x_2)
-(-1)^{|x_{2}||x_{4}|}\rho(\beta(x_2),\beta(x_4))f\alpha(x_1)\big)\\
&&+(-1)^{(|x_{3}|+|x_{4}|)(|x_{1}|+|x_{2}|)}f([\beta^2(x_3)\!,\!\beta^2(x_4)\!,\![\beta(x_1)\!,\!\beta(x_2)\!,\!\alpha(x_5)]])\!\\&&+(-1)^{|x_{3}|(|x_{1}|+|x_{2}|)}\rho(\beta^2(x_3)\!,\!\alpha^{-1}\beta([\beta(x_1)\!,\!\beta(x_2)\!,\!\alpha(x_5)]))f\alpha\beta(x_4)\\
&&-(-1)^{|x_{4}|(|x_{1}|+|x_{2}|)+|x_{3}|(|x_{4}|+|x_{5}|)}\rho(\beta^2(x_4),\alpha^{-1}\beta([\beta(x_1),\beta(x_2),\alpha(x_5)]))f\alpha\beta(x_3)\\
&&+(-1)^{(|x_{3}|+|x_{4}|)(|x_{1}|+|x_{2}|)}\rho(\beta^2(x_3),\beta^2(x_4))\big(-\rho(\beta(x_1),\beta(x_2))f\alpha(x_5)\\&&+(-1)^{|x_{2}||x_{5}|}\rho(\beta(x_1),\beta(x_5))f\alpha(x_2)
-(-1)^{|x_{1}|(|x_{2}||x_{5}|)}\rho(\beta(x_2),\beta(x_5))f\alpha(x_1)\big)\\
&=&(-1)^{(|x_{4}|+|x_{5}|)(|x_{1}|+|x_{2}|+|x_{3}|)}\theta(\beta^{2}(x_{4}),\beta^{2}(x_{5}),[\beta(x_{1}),\beta(x_{2}),\alpha(x_{3})])\\
&&+(-1)^{(|x_{4}|+|x_{5}|)(|x_{1}|+|x_{2}|+|x_{3}|)}\rho(\beta^{2}(x_{4}),\beta^{2}(x_{5}))\theta(\beta(x_{1}),\beta(x_{2}),\alpha(x_{3}))\\
&&-(-1)^{(|x_{3}|+|x_{5}|)(|x_{1}|+|x_{2}|)+|x_{4}||x_{5}|}\theta(\beta^{2}(x_{3}),\beta^{2}(x_{5}),[\beta(x_{1}),\beta(x_{2}),\alpha(x_{4})])\\
&&-(-1)^{(|x_{3}|+|x_{5}|)(|x_{1}|+|x_{2}|)+|x_{4}||x_{5}|}\rho(\beta^{2}(x_{3}),\beta^{2}(x_{5}))\theta(\beta(x_{1}),\beta(x_{2}),\alpha(x_{4}))\\
&&+(-1)^{(|x_{3}|+|x_{4}|)(|x_{1}|+|x_{2}|)}\theta(\beta^{2}(x_{3}),\beta^{2}(x_{4}),[\beta(x_{1}),\beta(x_{2}),\alpha(x_{5})])\\
&&+(-1)^{(|x_{3}|+|x_{4}|)(|x_{1}|+|x_{2}|)}\rho(\beta^{2}(x_{3}),\beta^{2}(x_{4}))\theta(\beta(x_{1}),\beta(x_{2}),\alpha(x_{5})).
\end{array}$

\vspace{0.5cm}
Then $\theta_{f}$ is a $3$-cocycle associated with $\rho$.
\end{proof}
\begin{cor}
Under the above notations, $\theta+\theta_{f}$ is a $3-$cocycle associated with $\rho$.
\end{cor}
\begin{prop}
Under the above notations, $\sigma:T_{\theta}(\mathfrak{g})\rightarrow T_{\theta+\theta_{f}}(\mathfrak{g})$ is a isomorphism of $3$-Bihom- Lie superalgebras, where
$\sigma(v+x)=v+f(v)+x,~\forall v\in\mathfrak{g},~x\in M$.
\end{prop}
\begin{proof}
It is clear that $\sigma$ is a bijection.\\
Let $v_{i}\in\mathfrak{g},~x_{i}\in M,~i=1,2,3,$

First, we have $\sigma\circ(\alpha+\alpha_{M})(v_{1}+x_{1})=\sigma(\alpha(v_{1})+\alpha_{M}(x_{1}))=\alpha(v_{1})
+f\circ\alpha(v_{1})+\alpha_{M}(x_{1})=\alpha(v_{1})+\alpha_{M}\circ f(v_{1})+\alpha_{M}(x_{1})=(\alpha+\alpha_{M})(v_{1}+f(v_{1})+x_{1})=(\alpha+\alpha_{M})\circ\sigma(v_{1}+x_{1}),$ then $\sigma\circ(\alpha+\alpha_{M})=(\alpha+\alpha_{M})\circ\sigma.$

 Similarly, $\sigma\circ(\beta+\beta_{M})=(\beta+\beta_{M})\circ\sigma.$

Now, we have

$$\begin{array}{lllllll}&&[\sigma(v_{1}+x_{1}),\sigma(v_{1}+x_{1}),\sigma(v_{1}+x_{1})]_{\theta+\theta_{f}}\\
&=&[v_{1}+f(v_{1})+x_{1},v_{2}+f(v_{2})+x_{2},v_{3}+f(v_{3})+x_{3}]_{\theta+\theta_{f}}\\
&=&[v_{1},v_{2},v_{3}]+(\theta+\theta_{f})(v_{1},v_{2},v_{3})+\rho(v_{1},v_{2})(f(v_{3})+x_{3})\\
&&-(-1)^{|v_{2}||v_{3}|}\rho(v_{1},\alpha^{-1}\beta(v_{3}))\alpha_{M}\beta_{M}^{-1}(f(v_{2})+x_{2})\\
&&+(-1)^{|v_{1}|(|v_{2}|+|v_{3}|)}\rho(v_{2},\alpha^{-1}\beta(v_{3}))\alpha_{M}\beta_{M}^{-1}(f(v_{1})+x_{1})\\
&=&[v_{1},v_{2},v_{3}]+\theta (v_{1},v_{2},v_{3})+f([v_{1},v_{2},v_{3}])\\&&-\rho(v_{1},v_{2})f(v_{3})+(-1)^{|v_{2}||v_{3}|}\rho(v_{1},\alpha^{-1}\beta(v_{3}))f\circ\alpha\beta^{-1}(v_{2})\\
&&-(-1)^{|v_{1}|(|v_{2}||v_{3}|)}\rho(v_{2},\alpha^{-1}\beta(v_{3}))f\circ\alpha\beta^{-1}(v_{1})\\
&&+\rho(v_{1},v_{2})f(v_{3})-(-1)^{|v_{2}||v_{3}|}\rho(v_{1},\alpha^{-1}\beta(v_{3}))\alpha_{M}\beta_{M}^{-1}f(v_{2})\\
&&+(-1)^{|v_{1}|(|v_{2}|+|v_{3}|)}\rho(v_{2},\alpha^{-1}\beta(v_{3}))\alpha_{M}\beta_{M}^{-1}f(v_{1})\\
&&+\rho(v_{1},v_{2})(x_{3})-(-1)^{|v_{2}||v_{3}|}\rho(v_{1},\alpha^{-1}\beta(v_{3}))\alpha_{M}\beta_{M}^{-1}(x_{2})\\
&&+(-1)^{|v_{1}|(|v_{2}|+|v_{3}|)}\rho(v_{2},\alpha^{-1}\beta(v_{3}))\alpha_{M}\beta_{M}^{-1}(x_{1})\\
&=&[v_{1},v_{2},v_{3}]+f([v_{1},v_{2},v_{3}])+\theta(v_{1},v_{2},v_{3})+\rho(v_{1},v_{2})(x_{3})\\
&&-(-1)^{|v_{2}||v_{3}|}\rho(v_{1},\alpha^{-1}\beta(v_{3}))\alpha_{M}\beta_{M}^{-1}(x_{2})
\end{array}$$
$$\begin{array}{lllll}
&&+(-1)^{|v_{1}|(|v_{2}|+|v_{3}|)}\rho(v_{2},\alpha^{-1}\beta(v_{3}))\alpha_{M}\beta_{M}^{-1}(x_{1})\\
&=&\sigma([v_{1},v_{2},v_{3}]+\theta(v_{1},v_{2},v_{3})+\rho(v_{1},v_{2})(x_{3})\\&&-(-1)^{|v_{2}||v_{3}|}\rho(v_{1},\alpha^{-1}\beta(v_{3}))\alpha_{M}\beta_{M}^{-1}(x_{2})\\
&&+(-1)^{|v_{1}|(|v_{2}|+|v_{3}|)}\rho(v_{2},\alpha^{-1}\beta(v_{3}))\alpha_{M}\beta_{M}^{-1}(x_{1}))\\
&=&\sigma([v_{1}+x_{1},v_{2}+x_{2},v_{3}+x_{3}]_{\theta}).
\end{array}$$
Then $\sigma:T_{\theta}(\mathfrak{g})\rightarrow T_{\theta+\theta_{f}}(\mathfrak{g})$ is a isomorphism of $3$-Bihom- Lie superalgebras.
\end{proof}
\section{$T^{\ast}_{\theta}$-EXTENSIONS OF $3$-BIHOM-LIE SUPERALGEBRAS}
The method of $T^{\ast}$-extension was introduced in \cite{B} and has already been used for $3$-Hom Lie
algebras in \cite{Lii} and $3$-Bihom-Lie algebras in \cite{Li C}.
\begin{defn}
Let $(\mathfrak{g},[.,.,.],\alpha,\beta)$ be a $3$-Bihom-Lie superalgebra. A bilinear form $f$ on $\mathfrak{g}$ is said to be nondegenerate if
$$\mathfrak{g}^{\perp}=\{x\in\mathfrak{g}\mid~f(x,y)=0,~\forall y\in\mathfrak{g}\}=\{0\};$$
$\alpha\beta-$invariant if  for all $x_{1},x_{2},x_{3},x_{4}\in\mathfrak{g},$
$$f([\beta(x_{1}),\beta(x_{2}),\alpha(x_{3})],\alpha(x_{4}))=f(\beta(x_{1}),[\beta(x_{2}),\alpha(x_{3}),\alpha(x_{4})]);$$
supersymmetric if for all $x,y\in\mathfrak{g}$,
$$f(x,y)=(-1)^{|x||y|}f(y,x).$$
A subspace $I$ of $\mathfrak{g}$ is called isotropic if $I\subseteq I^{\perp}$.
\end{defn}
\begin{defn}
Let $(\mathfrak{g},[.,.,.],\alpha,\beta)$ be a $3-$Bihom-Lie superalgebra over a field $\mathbb{K}$. If $\mathfrak{g}$ admits a nondegenerate, $\alpha\beta-$invariant and supersymmetric
bilinear form $f$ such that $\alpha,\beta$ are $f-$symmetric (i.e. $f(\alpha(x),y)=f(x,\alpha(y)),~f(\beta(x),y)=f(x,\beta(y))$), then we call $(\mathfrak{g},f,\alpha,\beta)$ a quadratic $3$-Bihom-Lie superalgebra.

Let $(\mathfrak{g}',[.,.,.]',\alpha',\beta')$ be another $3-$Bihom-Lie superalgebra. Two quadratic $3$-Bihom-Lie superalgebras $(\mathfrak{g},f,\alpha,\beta)$ and $(\mathfrak{g}',f',\alpha',\beta')$ are said to be isometric if there exists a superalgebra isomorphism $\phi:\mathfrak{g}\rightarrow\mathfrak{g}'$ such that $f(x,y)=f'(\phi(x),\phi(y)),~\forall x,y\in\mathfrak{g}$.
\end{defn}
\begin{thm}
Let $(\mathfrak{g},[.,.,.],\alpha,\beta)$ be a $3$-Bihom-Lie superalgebra and $(M,\rho,\alpha_{M},\beta_{M})$
be a
representation of $\mathfrak{g}$. Let us consider $M^{\ast}$ the dual space of $M$ and $\widetilde{\alpha}_{M},\widetilde{\beta}_{M}:M^{\ast}\rightarrow M^{\ast}$
two homomorphisms defined by $\widetilde{\alpha}_{M}(f)=f\circ\alpha_{M},~\widetilde{\beta}_{M}(f)=f\circ\beta_{M},~\forall f\in M^{\ast}$. Then
the superskewsymmetry linear map $\widetilde{\rho}:\mathfrak{g}\times\mathfrak{g}\rightarrow End(M^{\ast})$, defined by $\widetilde{\rho}(x,y)(f)=-(-1)^{|f|(|x|+|y|)}f\circ\rho(x,y),~\forall f\in M^{\ast},
x,y\in\mathfrak{g}$,
is a representation of $\mathfrak{g}$ on $(M^{\ast},\widetilde{\rho},\widetilde{\alpha}_{M},\widetilde{\beta}_{M})$ if and only if for
every $x, y, u,v\in\mathfrak{g}$,
\begin{enumerate}
\item
$\quad\alpha_{M}\circ\rho(\alpha(x),\alpha(y))=\rho(x,y)\circ\alpha_{M},$
\item
$\quad\beta_{M}\circ\rho(\beta(x),\beta(y))=\rho(x,y)\circ\beta_{M},$
\item
$\quad\rho(x,y)\rho(\alpha\beta(u),\alpha\beta(v))$\\
$=(-1)^{(|x|+|y|)(|u|+|v|)}\rho(\alpha(u),\alpha(v))\rho(\beta(x),\beta(y))-(-1)^{(|x|+|y|)(|u|+|v|)}\beta_{M}\rho([\beta(u),\beta(v),x],\beta(y))$\\
$~~~-(-1)^{|y|(|u|+|v|)}\beta_{M}\rho(\beta(x),[\beta(u),\beta(v),y]),$
\item
$\quad\beta_M\rho([\beta(u),\beta(v),x],\beta(y))\\
    =-(-1)^{|y|(|v|+|x|)}\rho(\alpha(u),y)\rho(\alpha\beta(v),\beta(x))\!-(-1)^{|u|(|x|+|y|+|v|)+|x||y|}\!\rho(\alpha(v),y)\rho(\beta(x),\alpha\beta(u))\!\\
    ~~~-(-1)^{(|x|+|y|)(|u|+|v|)}\!\rho(x,y)\rho(\alpha\beta(u),\alpha\beta(v)).$
\end{enumerate}
\end{thm}

\begin{proof}Let $f \in M^*$, $x,y,u,v \in\mathfrak{g}$.
First, we have
$$
(\tilde\rho(\alpha(u),\alpha(v))\circ \tilde{\alpha}_M)(f)=- (-1)^{|f|(|u|+|v|)}\tilde{\alpha}_M(f)\circ \rho(\alpha(u),\alpha(v))=-(-1)^{|f|(|u|+|v|)}f\circ \alpha_M\circ \rho(\alpha(u),\alpha(v))$$ and
$\tilde{\alpha}_M\circ\tilde\rho(u,v)(f)=-(-1)^{|f|(|u|+|v|)}\tilde{\alpha}_M(f\circ\rho(u,v))=-(-1)^{|f|(|u|+|v|)}f\circ\rho(u,v)\circ\alpha_M,$ which implies $$\tilde\rho(\alpha(u),\alpha(v))\circ \tilde{\alpha}_M=\tilde{\alpha}_M\circ\tilde\rho(u,v)\Leftrightarrow \alpha_M\circ \rho(\alpha(u),\alpha(v))=\rho(u,v)\circ\alpha_M.$$
Similarly, $\tilde\rho(\beta(u),\beta(v))\circ \tilde{\beta}_M=\tilde{\beta}_M\circ\tilde\rho(u,v)\Leftrightarrow \beta_M\circ \rho(\beta(u),\beta(v))=\rho(u,v)\circ\beta_M.$

Then we can get\\

$
\begin{array}{llll}\tilde\rho(\alpha\beta(u),\alpha\beta(v))\circ\tilde\rho(x,y)(f)&=&-(-1)^{|f|(|x|+|y|)}\tilde\rho(\alpha\beta(u),\alpha\beta(v))(f\rho(x,y))\\
&=&(-1)^{|f|(|x|+|y|+|u|+|v|)+(|u|+|v|)(|x|+|y|)}f\rho(x,y)\rho(\alpha\beta(u),\alpha\beta(v))\end{array}$

 and
\begin{eqnarray*}
&&\big(\tilde\rho(\beta(x),\beta(y))\circ\tilde\rho(\alpha(u),\alpha(v))+\tilde\rho([\beta(u),\beta(v),x],\beta(y))\circ\tilde\beta_M\\
&&\,+\tilde\rho(\beta(x),[\beta(u),\beta(v),y])\circ\tilde\beta_M\big)(f)\\
&=&(-1)^{|f|(|x|+|y|+|u|+|v|)+(|u|+|v|)(|x|+|y|)}\!f\rho(\alpha(u),\alpha(v)\!)\!\rho(\beta(x),\beta(y)\!)\!\!\\
&&-(-1)^{|f|(|x|+|y|+|u|+|v|)}\!f\beta_M\rho([\beta(u),\beta(v),x],\beta(y)\!)\!\!\\
&&-(-1)^{|f|(|x|+|y|+|u|+|v|)}\!f\beta_M\rho(\beta(x),[\beta(u),\beta(v),y]),
\end{eqnarray*}
which implies
\begin{eqnarray*}
&&\tilde\rho(\alpha\beta(u),\alpha\beta(v))\circ\tilde\rho(x,y)\\
&=&(-1)^{(|x|+|y|)(|u|+|v|)}\tilde\rho(\beta(x),\!\beta(y)\!)\!\circ\!\tilde\rho(\alpha(u),\!\alpha(v)\!)\!+\!\tilde\rho([\beta(u),\!\beta(v),\!x],\!\beta(y)\!)\!\circ\!\tilde\beta_M\\
\!&&+\!(-1)^{|x|(|u|+|v|)}\tilde\rho(\beta(x),[\beta(u),\beta(v),y])\!\circ\!\tilde\beta_M
\end{eqnarray*}
if and only if
\begin{eqnarray*}
&&\rho(x,y)\rho(\alpha\beta(u),\alpha\beta(v))\\
&=&(-1)^{(|x|+|y|)(|u|+|v|)}\rho(\alpha(u),\alpha(v))\rho(\beta(x),\beta(y))-(-1)^{(|x|+|y|)(|u|+|v|)}\beta_{M}\rho([\beta(u),\beta(v),x],\beta(y))\\
&&-(-1)^{|y|(|u|+|v|)}\beta_{M}\rho(\beta(x),[\beta(u),\beta(v),y]),
\end{eqnarray*}
In the same way,
\begin{eqnarray*}
&&\tilde\rho([\beta(u),\beta(v),x],\beta(y))\tilde\beta_M\\
&=&(-1)^{|u|(|x|+|v|)}\tilde\rho(\alpha\beta(v),\beta(x)\!)\!\circ\!\tilde\rho(\alpha(u),y\!)\!\!+\!(-1)^{|x|(|u|+|v|)}\tilde\rho(\beta(x),\alpha\beta(u)\!)\!\circ\!\tilde\rho(\alpha(v),y)\\
\!&&+\!\tilde\rho(\alpha\beta(u),\alpha\beta(v))\!\circ\!\tilde\rho(x,y)
\end{eqnarray*}
if and only if\\

$\beta_M\rho([\beta(u),\beta(v),x],\beta(y))\\
    =-(-1)^{|y|(|v|+|x|)}\rho(\alpha(u),y)\rho(\alpha\beta(v),\beta(x))\!-(-1)^{|u|(|x|+|y|+|v|)+|x||y|}\!\rho(\alpha(v),y)\rho(\beta(x),\alpha\beta(u))\\
    ~~~~-(-1)^{(|x|+|y|)(|u|+|v|)}\!\rho(x,y)\rho(\alpha\beta(u),\alpha\beta(v)).$\\

That shows the theorem holds.
\end{proof}
\begin{cor}
Let $\ad$ be the adjoint representation of a $3$-Bihom-Lie superalgebra $(\mathfrak{g},[\cdot,\cdot,\cdot],\\
\alpha,\beta)$. Let us consider the bilinear map $\ad^*:\mathfrak{g}\times \mathfrak{g}\rightarrow \mathrm{End}(\mathfrak{g}^{*})$ defined by $$\ad^*(x,y)(f)=-(-1)^{|f|(|x|+|y|)}f\circ \ad(x,y),~\forall\, x,y\in \mathfrak{g}.$$ Then $\ad^*$ is a representation of $\mathfrak{g}$ on $(\mathfrak{g}^{*},\ad^*,\tilde{\alpha},\tilde\beta)$ if and only if
\begin{enumerate}
\item $\alpha\circ \ad(\alpha(x),\alpha(y))=\ad(x,y)\circ \alpha,$
\item $\beta\circ \ad(\beta(x),\beta(y))=\ad(x,y)\circ \beta,$
\item
$ $
$\begin{array}{llllll}&&ad(x,y)\ad(\alpha\beta(u),\alpha\beta(v))\\
&=&(-1)^{(|x|+|y|)(|u|+|v|)}\ad(\alpha(u),\alpha(v))\ad(\beta(u),\beta(v))\!\!-(-1)^{(|x|+|y|)(|u|+|v|)}\beta\ad([\beta(u),\beta(v),x],\beta(y))\\
&&-(-1)^{|y|(|u|+|v|)}\beta\ad(\beta(x),[\beta(u),\beta(v),y]),\end{array}$
\item $\quad\beta\ad([\beta(u),\beta(v),x],\beta(y))$\\
    $=\!-(-1)^{|y|(|v|+|x|)}\ad(\alpha(u),y)\ad(\alpha\beta(v),\beta(x)\!)-(-1)^{|u|(|x|+|y|+|v|)+|x||y|}\ad(\alpha(v),y)\ad(\beta(x),\alpha\beta(u)\!)$\\
    $~~~-(-1)^{(|x|+|y|)(|u|+|v|)}\ad(x,y)\ad(\alpha\beta(u),\alpha\beta(v)\!),$
\end{enumerate}
We call the representation $\ad^*$ the coadjoint representation of $\mathfrak{g}$.
\end{cor}
\begin{defn}
Let $\mathfrak{g}$ be a $3$-Bihom-Lie superalgebra over a field $\mathbb{K}$. We inductively define a derived series
$$(\mathfrak{g}^{(n)})_{n\geq 0}: \mathfrak{g}^{(0)}=\mathfrak{g},\ \mathfrak{g}^{(n+1)}=[\mathfrak{g}^{(n)},\mathfrak{g}^{(n)},\mathfrak{g}]$$
and a central descending series
$$(\mathfrak{g}^{n})_{n\geq 0}: \mathfrak{g}^{0}=\mathfrak{g},\ \mathfrak{g}^{n+1}=[\mathfrak{g}^{n},\mathfrak{g},\mathfrak{g}].$$

$\mathfrak{g}$ is called solvable and nilpotent $($of length $k$$)$ if and only if there is a $($smallest$)$ integer $k$ such that $\mathfrak{g}^{(k)}=0$ and $\mathfrak{g}^{k}=0$, respectively.
\end{defn}

\begin{thm}
Let $(\mathfrak{g},[\cdot,\cdot,\cdot],\alpha,\beta)$ be a $3$-Bihom-Lie superalgebra over a field $\mathbb{K}$.
\begin{enumerate}
   \item  If $\mathfrak{g}$ is solvable, then $(\mathfrak{g}\oplus \mathfrak{g}^{*},[\cdot,\cdot,\cdot]_\theta,\alpha+\tilde{\alpha},\beta+\tilde{\beta})$ is solvable.
   \item  If $\mathfrak{g}$ is nilpotent, then $(\mathfrak{g}\oplus \mathfrak{g}^{*},[\cdot,\cdot,\cdot]_\theta,\alpha+\tilde{\alpha},\beta+\tilde{\beta})$ is nilpotent.
\end{enumerate}
\end{thm}
\begin{proof}
\begin{enumerate}
\item
We suppose that $\mathfrak{g}$ is solvable of length $s$, i.e. $\mathfrak{g}^{(s)}=[\mathfrak{g}^{(s-1)},\mathfrak{g}^{(s-1)},\mathfrak{g}]=0.$ We claim that $(\mathfrak{g}\oplus \mathfrak{g}^*)^{(k)}\subseteq \mathfrak{g}^{(k)}+\mathfrak{g}^*$, which we prove by induction on $k$. The case $k=1$, by Proposition \ref{propk}, we have
\begin{eqnarray*}
(\mathfrak{g}\oplus \mathfrak{g}^*)^{(1)}&=&[\mathfrak{g}\oplus \mathfrak{g}^*,\mathfrak{g}\oplus \mathfrak{g}^*,\mathfrak{g}\oplus \mathfrak{g}^*]_\theta\\
&=&[\mathfrak{g},\mathfrak{g},\mathfrak{g}]_\theta+[\mathfrak{g},\mathfrak{g},\mathfrak{g}^*]_\theta+[\mathfrak{g},\mathfrak{g}^*,\mathfrak{g}]_\theta+[\mathfrak{g}^*,\mathfrak{g},\mathfrak{g}]_\theta\\
&=&[\mathfrak{g},\mathfrak{g},\mathfrak{g}]+\theta(\mathfrak{g},\mathfrak{g},\mathfrak{g})+[\mathfrak{g},\mathfrak{g},\mathfrak{g}^*]_\theta+[\mathfrak{g},\mathfrak{g}^*,\mathfrak{g}]_\theta+[\mathfrak{g}^*,\mathfrak{g},\mathfrak{g}]_\theta\\
&\subseteq&\mathfrak{g}^{(1)}+\mathfrak{g}^*.
\end{eqnarray*}
By induction, $(\mathfrak{g}\oplus \mathfrak{g}^*)^{(k-1)}\subseteq \mathfrak{g}^{(k-1)}+\mathfrak{g}^*$. So
\begin{eqnarray*}
&&(\mathfrak{g}\oplus \mathfrak{g}^*)^{(k)}\\
&=&[(\mathfrak{g}\oplus \mathfrak{g}^*)^{(k-1)},(\mathfrak{g}\oplus \mathfrak{g}^*)^{(k-1)},\mathfrak{g}\oplus \mathfrak{g}^*]_\theta\\
&\subseteq&[\mathfrak{g}^{(k-1)}+\mathfrak{g}^*,\mathfrak{g}^{(k-1)}+\mathfrak{g}^*,\mathfrak{g}\oplus \mathfrak{g}^*]_\theta\\
&=&[\mathfrak{g}^{(k-1)},\mathfrak{g}^{(k-1)},\mathfrak{g}]+\theta(\mathfrak{g}^{(k-1)},\mathfrak{g}^{(k-1)},\mathfrak{g})+[\mathfrak{g}^{(k-1)},\mathfrak{g}^{(k-1)},\mathfrak{g}^*]_\theta+[\mathfrak{g}^{(k-1)},\mathfrak{g}^*,\mathfrak{g}]_\theta\\
&&+[\mathfrak{g}^*,\mathfrak{g}^{(k-1)},\mathfrak{g}]_\theta\\
&\subseteq&\mathfrak{g}^{(k)}+\mathfrak{g}^*.
\end{eqnarray*}
Therefore
\begin{eqnarray*}
&&(\mathfrak{g}\oplus \mathfrak{g}^*)^{(s+1)}\\
&\subseteq&[\mathfrak{g}^{(s)},\mathfrak{g}^{(s)},\mathfrak{g}]+\theta(\mathfrak{g}^{(s)},\mathfrak{g}^{(s)},\mathfrak{g})+[\mathfrak{g}^{(s)},\mathfrak{g}^{(s)},\mathfrak{g}^*]_\theta+[\mathfrak{g}^{(s)},\mathfrak{g}^*,\mathfrak{g}]_\theta+[\mathfrak{g}^*,\mathfrak{g}^{(s)},\mathfrak{g}]_\theta\\
&=&0.
\end{eqnarray*}
It follows $(\mathfrak{g}\oplus \mathfrak{g}^{*},[\cdot,\cdot,\cdot]_\theta,\alpha+\tilde{\alpha},\beta+\tilde{\beta})$ is solvable.
\item
Suppose that $\mathfrak{g}$ is nilpotent of length $s$. Since $(\mathfrak{g}\oplus \mathfrak{g}^{*})^{s}/\mathfrak{g}^{*}\cong \mathfrak{g}^{s}$ and $\mathfrak{g}^{s}=0$, we have
$(\mathfrak{g}\oplus \mathfrak{g}^{*})^{s}\subseteq \mathfrak{g}^{*}$. Let $h\in(\mathfrak{g}\oplus \mathfrak{g}^{*})^{s}\subseteq \mathfrak{g}^{*},~ b\in \mathfrak{g},~x_{i}+f_{i},~ y_{i}+g_{i}\in \mathfrak{g}\oplus \mathfrak{g}^{*}, ~1\leq i\leq s-1$, we have
$$\begin{array}{lllllllll}
&&[[\cdots[h,x_1+f_{1},y_{1}+g_{1}]_\theta,\cdots]_\theta,x_{s-1}+f_{s-1},y_{s-1}+g_{s-1}]_\theta(b)\\
&=&(-1)^{(|x_{1}|+|y_{1}|)|h|+(|x_{2}|+|y_{2}|)(|x_{1}|+|y_{1}|+|h|)+\cdots+(|x_{s}|+|y_{s}|)(|x_{1}|+|y_{1}|+\cdots+\cdots|x_{s-1}|+|y_{s-1}|+|h|)}\\
&& \times h\alpha\beta^{-1}\ad(x_1,\!\alpha^{-1}\beta(y_1)\!)\alpha\beta^{-1}\ad(x_2,\!\alpha^{-1}\beta(y_2)\!)\!\cdots\!\alpha\beta^{-1}\ad(x_{s-1},\!\alpha^{-1}\beta(y_{s-1})\!)\!(b)\\
&=&(-1)^{(|x_{1}|+|y_{1}|)|h|+(|x_{2}|+|y_{2}|)(|x_{1}|+|y_{1}|+|h|)+\cdots+(|x_{s}|+|y_{s}|)(|x_{1}|+|y_{1}|+\cdots+\cdots|x_{s-1}|+|y_{s-1}|+|h|)}\\
&&\times h\alpha\beta^{-1}([x_1,\alpha^{-1}\beta(y_1),\alpha^{-1}\beta[x_2,\alpha^{-1}\beta(y_2),\!\cdots\!,\alpha\beta^{-1}[x_{s-1},\alpha^{-1}\beta(y_{s-1}),b]\cdots]])\\
& &\in h(\mathfrak{g}^{s})=0.
\end{array}$$
Thus $(\mathfrak{g}\oplus \mathfrak{g}^{*},[\cdot,\cdot,\cdot]_\theta,\alpha+\tilde{\alpha},\beta+\tilde{\beta})$ is nilpotent.
\end{enumerate}
\end{proof}

Now we consider the following symmetric bilinear form $q_{\mathfrak{g}}$ on $\mathfrak{g}\oplus \mathfrak{g}^{*}$,
$$q_{\mathfrak{g}}(x+f,y+g)=f(y)+(-1)^{|x||y|}g(x),~ \forall\, x+f, y+g\in \mathfrak{g}\oplus \mathfrak{g}^*.$$
Obviously, $q_{\mathfrak{g}}$ is  nondegenerate. In fact, if $x+f$ is orthogonal to  all elements $y+g$ of $\mathfrak{g}\oplus \mathfrak{g}^{*}$, then $f(y)=0$ and $(-1)^{|x||y|}g(x)=0$,  which implies that $x=0$ and $f=0$.

\begin{lem}\label{lemma3.2}
Let $q_\mathfrak{g}$ be as above. Then the 4-tuple $(\mathfrak{g}\oplus \mathfrak{g}^{*},q_\mathfrak{g},\alpha+\tilde{\alpha},\beta+\tilde{\beta})$ is a quadratic $3$-Bihom-Lie superalgebra if and only if $\theta$ satisfies for all $ x_1,x_2,x_3,x_4\in \mathfrak{g}$,
\begin{eqnarray*}\label{203}
\theta(\beta(x_1),\beta(x_2),\alpha(x_3))(\alpha(x_4))+(-1)^{|x_{3}||x_{4}|}\theta(\beta(x_1),\beta(x_2),\alpha(x_4))(\alpha(x_3))
=0.
\end{eqnarray*}
\end{lem}
\begin{proof}

Now suppose that  $x_i+f_i\in \mathfrak{g}\oplus \mathfrak{g}^*, i=1,2,3,4$, we have
\begin{eqnarray*}
q_{\mathfrak{g}}((\alpha+\tilde{\alpha})(x_1+f_1),x_2+f_2)&=&q_\mathfrak{g}(\alpha(x_1)+f_1\circ\alpha,x_2+f_2)\\
&=&f_1(\alpha(x_2))+(-1)^{|x_{1}||x_{2}|}f_2\circ\alpha(x_1)\\
&=&q_\mathfrak{g}(x_1+f_1,(\alpha+\tilde{\alpha})(x_2+f_2)).
\end{eqnarray*}
Then $\alpha+\tilde{\alpha}$ is $q_\mathfrak{g}$-symmetric. In the same way, $\beta+\tilde{\beta}$ is $q_\mathfrak{g}$-symmetric.

Next, we can obtain\\\\
$\begin{array}{lllll}
&&q_{\mathfrak{g}}\big([(\beta+\tilde{\beta})(x_1+f_1),(\beta+\tilde{\beta})(x_2+f_2), (\alpha+\tilde{\alpha})(x_3+f_3)]_\theta, (\alpha+\tilde{\alpha})(x_4+f_4)\big)\\
&&+(-1)^{|x_{3}|(|x_{1}|+|x_{2}|)}q_{\mathfrak{g}}\big((\alpha+\tilde{\alpha})(x_3+f_3),[(\beta+\tilde{\beta})(x_1+f_1), (\beta+\tilde{\beta})(x_2+f_2), (\alpha+\tilde{\alpha})(x_4+f_4)]_\theta\big)\\
&=&q_{\mathfrak{g}}\big([\beta(x_1)+f_1\circ\beta, \beta(x_2)+f_2\circ\beta, \alpha(x_3)+f_3\circ\alpha]_\theta, \alpha(x_4)+f_4\circ\alpha\big)\\
&&+(-1)^{|x_{3}|(|x_{1}|+|x_{2}|)}q_{\mathfrak{g}}\big(\alpha(x_3)+f_3\circ\alpha, [\beta(x_1)+f_1\circ\beta, \beta(x_2)+f_2\circ\beta, \alpha(x_4)+f_4\circ\alpha]_\theta\big)\\
&=&q_{\mathfrak{g}}\big([\beta(x_1),\beta(x_2),\alpha(x_3)]+\theta(\beta(x_1),\beta(x_2),\alpha(x_3))+\ad^*(\beta(x_1),\beta(x_2))(f_3\circ\alpha)
\end{array}$
$\begin{array}{lllll}
&&-(-1)^{|x_{2}||x_{3}|}\ad^*(\beta(x_1),\alpha^{-1}\beta\alpha(x_3))\tilde{\alpha}\tilde{\beta}^{-1}(f_2\circ\beta)\\
&&+(-1)^{|x_{1}|(|x_{2}|+|x_{3}|)}\ad^*(\beta(x_2),\alpha^{-1}\beta\alpha(x_3))\tilde{\alpha}\tilde{\beta}^{-1}(f_1\circ\beta),\\
&&\alpha(x_4)+f_4\circ\alpha\big)+(-1)^{|x_{3}|(|x_{1}|+|x_{2}|)}q_{\mathfrak{g}}\big(\alpha(x_3)+f_3\circ\alpha,[\beta(x_1),\beta(x_2),\alpha(x_4)]\\&&+\theta(\beta(x_1),\beta(x_2),\alpha(x_4))
+\ad^*(\beta(x_1),\beta(x_2))(f_4\circ\alpha)-(-1)^{|x_{2}||x_{4}|}\ad^*(\beta(x_1),\\&&\alpha^{-1}\beta\alpha(x_4))\tilde{\alpha}\tilde{\beta}^{-1}(f_2\circ\beta)
+(-1)^{|x_{1}|(|x_{2}|+|x_{4}|)}\ad^*(\beta(x_2),\alpha^{-1}\beta\alpha(x_4))\tilde{\alpha}\tilde{\beta}^{-1}(f_1\circ\beta)\big)\\
&=&\theta(\beta(x_1),\beta(x_2),\alpha(x_3)\!)(\alpha(x_4)\!)\!-\!(-1)^{|f_{3}|(|x_{1}|+|x_{2}|)}f_3\alpha([\beta(x_1),\beta(x_2),\alpha(x_4)])\\
&&+(-1)^{|x_{2}||x_{3}|+|f_{2}|(|x_{1}|+|x_{3}|)}f_2\alpha([\beta(x_1),\beta(x_3),\alpha(x_4)])\\
&&-(-1)^{(|x_{1}|+|f_{1}|)(|x_{2}|+|x_{3}|)}f_1\alpha([\beta(x_2),\beta(x_3),\alpha(x_4)])\!+\!f_4\alpha([\beta(x_1),\beta(x_2),\alpha(x_3)])\\
&&+(-1)^{|x_{3}|(|x_{1}|+|x_{2}|+|x_{4}|)}(-1)^{|x_{3}|(|x_{1}|+|x_{2}|)}\theta(\beta(x_1),\beta(x_2),\alpha(x_4)\!)\!(\alpha(x_3)\!)\!\\
&&-(-1)^{|f_{4}|(|x_{1}|+|x_{2}|)}f_4\alpha([\beta(x_1),\beta(x_2),\alpha(x_3)])\\
&&+(-1)^{|x_{3}|(|x_{2}|+|x_{4}|)+|f_{2}|(|x_{1}|+|x_{3}|)}f_2\alpha([\beta(x_1),\beta(x_4),\alpha(x_3)])\\
&&-(-1)^{(|x_{1}|+|f_{1}|)(|x_{2}|+|x_{3}|)+|x_{3}||x_{4}|}f_1\alpha([\beta(x_2),\beta(x_4),\alpha(x_3)])\\
&&+(-1)^{|f_{3}|(|x_{1}|+|x_{2}|)}f_3\alpha([\beta(x_1),\beta(x_2),\alpha(x_4)])\\
&=&\theta(\beta(x_1),\beta(x_2),\alpha(x_3))(\alpha(x_4))+(-1)^{|x_{3}|x_{4}||}\theta(\beta(x_1),\beta(x_2),\alpha(x_4))(\alpha(x_3)),
\end{array}$\\

which implies
\begin{eqnarray*}
&&q_{\mathfrak{g}}\big([(\beta+\tilde{\beta})(x_1+f_1),(\beta+\tilde{\beta})(x_2+f_2), (\alpha+\tilde{\alpha})(x_3+f_3)]_\theta, (\alpha+\tilde{\alpha})(x_4+f_4)\big)\\
&+&\!\!\!\!(-1)^{|x_{3}|(|x_{1}|+|x_{2}|)}q_{\mathfrak{g}}\big((\alpha+\tilde{\alpha})(x_3+f_3),[(\beta+\tilde{\beta})(x_1+f_1), (\beta+\tilde{\beta})(x_2+f_2), (\alpha+\tilde{\alpha})(x_4+f_4)]_\theta\big)=0
\end{eqnarray*}
if and only if
$\theta(\beta(x_1),\beta(x_2),\alpha(x_3))(\alpha(x_4))+(-1)^{|x_{3}|x_{4}||}\theta(\beta(x_1),\beta(x_2),\alpha(x_4))(\alpha(x_3))
=0$.

Hence the lemma follows.
\end{proof}

Now, we shall call the quadratic $3$-Bihom-Lie superalgebra $(\mathfrak{g}\oplus \mathfrak{g}^{*},q_\mathfrak{g},\alpha+\tilde{\alpha},\beta+\tilde{\beta})$ the $T^*_\theta$-extension of $\mathfrak{g}$ (by $\theta$) and denote by $T_\theta^*(\mathfrak{g})$.
\begin{lem}\label{lemma3.1}
Let $(\mathfrak{g},q_\mathfrak{g},\alpha,\beta)$ be a $2n$-dimensional quadratic $3$-Bihom-Lie superalgebra over a field $\mathbb{K}$ $(ch\mathbb{K}\neq2)$, $\alpha$ be surjective and $I$ be an isotropic $n$-dimensional subspace of $\mathfrak{g}$. If $I$ is a Bihom-ideal of $(\mathfrak{g},[\cdot,\cdot,\cdot],\alpha,\beta)$, then $[\beta(I),\beta(\mathfrak{g}),\alpha(\mathfrak{g})]=0$.
\end{lem}
\begin{proof}
Since dim$I$+dim$I^{\bot}=n+\dim I^{\bot}=2n$ and $I\subseteq I^{\bot}$, we have $I=I^{\bot}$.
If $I$ is a Bihom-ideal of $(\mathfrak{g},[\cdot,\cdot],\alpha,\beta)$, then \begin{eqnarray*}
q_\mathfrak{g}([\beta(I),\beta(\mathfrak{g}),\alpha(I^{\bot})],\alpha(\mathfrak{g}))&=&\pm q_\mathfrak{g}(\alpha(I^{\bot}),[\beta(I),\beta(\mathfrak{g}),\alpha(\mathfrak{g})])\\
&\subseteq&q_\mathfrak{g}(\alpha(I^{\bot}),[I,\beta(\mathfrak{g}),\alpha(\mathfrak{g})])\\
&\subseteq&q_\mathfrak{g}(I^{\bot},I)=0,
\end{eqnarray*}
which implies $[\beta(I),\beta(\mathfrak{g}),\alpha(I)]=[\beta(I),\beta(\mathfrak{g}),\alpha(I^{\bot})]\subseteq \alpha(\mathfrak{g})^{\bot}=\mathfrak{g}^{\bot}=0$.
\end{proof}

\begin{thm}
Let $(\mathfrak{g},q_\mathfrak{g},\alpha,\beta)$ be a quadratic regular $3$-Bihom-Lie superalgebra of dimensional $2n$ over a field $\mathbb{K}$ $(ch\mathbb{K}\neq2)$. Then $(\mathfrak{g},q_\mathfrak{g},\alpha,\beta)$ is isometric to a $T^{*}_\theta$-extension $(T_{\theta}^{*}(B),q_{B},\alpha^{'},\beta^{'})$ if and only if $(\mathfrak{g},[\cdot,\cdot,\cdot],\alpha,\beta)$ contains an isotropic Bihom-ideal $I$ of dimension $n$. In particular, $B\cong \mathfrak{g}/I$.
\end{thm}
\begin{proof}
($\Longrightarrow$) Suppose $\phi:B\oplus B^*\rightarrow \mathfrak{g}$ is isometric, we have $\phi(B^*)$ is a $n$-dimensional isotropic Bihom-ideal of $\mathfrak{g}$. In fact, since $\phi$ is isometric, ${\rm dim}B\oplus B^*={\rm dim}\mathfrak{g}=2n$, which implies ${\rm dim}B^*={\rm dim}\phi(B^*)=n$. And $0=q_B(B^*,B^*)=q_\mathfrak{g}(\phi(B^*),\phi(B^*))$, we have $\phi(B^*)\subseteq\phi(B^*)^{\bot}$. By $[\phi(B^*),\mathfrak{g},\mathfrak{g}]=[\phi(B^*),\phi(B\oplus B^*),\phi(B\oplus B^*)]=\phi([B^*,B\oplus B^*,B\oplus B^*]_\theta)\subseteq\phi(B^*)$, $\phi(B^*)$ is a Bihom-ideal of $\mathfrak{g}$. Furthermore, $B\cong B\oplus B^*/B^*\cong \mathfrak{g}/\phi(B^*)$.

($\Longleftarrow$) Suppose that $I$ is a $n$-dimensional isotropic $3$- Bihom-ideal of $\mathfrak{g}$. By Lemma \ref{lemma3.1},  $[\beta(I),\beta(\mathfrak{g}),\alpha(I)]=0$. Let $B=\mathfrak{g}/I$ and $p: \mathfrak{g} \rightarrow B$ be the canonical projection. We can choose an isotropic complement subspace $B_{0}$ to $I$ in $\mathfrak{g}$, i.e. $\mathfrak{g}=B_{0}\dotplus I$ and $B_{0}\subseteq B_{0}^{\bot}$. Then $B_{0}^{\bot}=B_{0}$ since dim$B_0=n$.

Denote by $p_{0}$ (resp. $p_{1}$) the projection $\mathfrak{g}=B_{0}\dotplus I \rightarrow B_{0}$ (resp. $\mathfrak{g}=B_{0}\dotplus I\rightarrow I$) and let $q_{\mathfrak{g}}^{*}:I \rightarrow B^{*}$ is a linear map, where $q_{\mathfrak{g}}^{*}(i)(\bar{x}):= q_{\mathfrak{g}}(i,x),~\forall\, i\in I, \bar{x}\in B=\mathfrak{g}/I$.
 We claim that $q_{\mathfrak{g}}^{*}$ is a vector space isomorphism. In fact, if $\bar{x}=\bar{y}$, then $x-y\in I$, hence $q_{\mathfrak{g}}(i,x-y)\in q_{\mathfrak{g}}(I,I)=0$ and
 so $q_{\mathfrak{g}}(i,x)=q_{\mathfrak{g}}(i,y)$, which implies $q_{\mathfrak{g}}^{*}$ is well-defined and it is easy to see that $q_{\mathfrak{g}}^{*}$ is linear. If
 $q_{\mathfrak{g}}^{*}(i)=q_{\mathfrak{g}}^{*}(j)$, then $q_{\mathfrak{g}}^{*}(i)(\bar{x})=q_{\mathfrak{g}}^{*}(j)(\bar{x}), ~\forall\, x\in \mathfrak{g}$, i.e. $q_{\mathfrak{g}}(i,x)=q_{\mathfrak{g}}(j,x)$,
 which implies $i-j\in \mathfrak{g}^\bot=0$, hence $q_{\mathfrak{g}}^{*}$ is injective. Note that $\dim I=\dim B^*=n$, then $q_{\mathfrak{g}}^{*}$ is surjective.

In addition, $q_{\mathfrak{g}}^{*}$ has the following property, $\forall\, x,y,z\in \mathfrak{g}, i\in I$,
\begin{eqnarray*}
q_{\mathfrak{g}}^{*}([\beta(x),\beta(y),\alpha(i)])(\bar{\alpha}(\bar{z}))
&=&q_{\mathfrak{g}}([\beta(x),\beta(y),\alpha(i)],\alpha(z))\\
&=&(-1)^{|i|(|x|+|y|)}q_{\mathfrak{g}}(\alpha(i),[\beta(x),\beta(y),\alpha(z)])\\
&=&(-1)^{|i|(|x|+|y|)}q_{\mathfrak{g}}^{*}(\alpha(i))(\overline{[\beta(x),\beta(y),\alpha(z)]})\\
&=&(-1)^{|i|(|x|+|y|)}q_{\mathfrak{g}}^{*}(\alpha(i))([\overline{\beta(x)},\overline{\beta(y)},\overline{\alpha(z)}])\\
&=&(-1)^{|i|(|x|+|y|)}q_{\mathfrak{g}}^{*}(\alpha(i))\ad(\overline{\beta(x)},\overline{\beta(y)})(\overline{\alpha(z)})\\
&=&\ad^*(\overline{\beta(x)},\overline{\beta(y)})q_{\mathfrak{g}}^{*}(\alpha(i))(\overline{\alpha(z)}).
\end{eqnarray*}
A similar computation shows that
$$q_{\mathfrak{g}}^{*}([\beta(x),\beta(i),\alpha(y)])=-(-1)^{|y||i|}\ad^*(\overline{\beta(x)},\overline{\beta(y)})q_{\mathfrak{g}}^{*}(\alpha(i)),$$
$$q_{\mathfrak{g}}^{*}([\beta(i),\beta(x),\alpha(y)])=(-1)^{(|x|+|y|)|i|}\ad^*(\overline{\beta(x)},\overline{\beta(y)})q_{\mathfrak{g}}^{*}(\alpha(i)).$$
Define a $3$-linear map
\begin{eqnarray*}
\theta:~~~B\times B\times B&\longrightarrow&B^{*}\\
(\bar{b_1},\bar{b_2},\bar{b_3})&\longmapsto&q_{\mathfrak{g}}^{*}(p_{1}([b_1,b_2,b_3])),
\end{eqnarray*}
where $b_1,b_2,b_3\in B_{0}.$ Then $\theta$ is well-defined since $p|_{B_0}$ is a vector space isomorphism.

Now define the bracket $[\cdot,\cdot,\cdot]_\theta$ on $B\oplus B^*$ by Proposition \ref{propk}, we have $B\oplus B^*$
is a algebra. Let $\varphi:\mathfrak{g} \rightarrow B\oplus B^{*}$ be a linear map defined by $\varphi(x+i)=\bar{x}+q_{\mathfrak{g}}^{*}(i),~ \forall\, x+i\in B_0\dotplus I=\mathfrak{g}. $ Since $p|_{B_0}$ and $q_{\mathfrak{g}}^{*}$ are vector space isomorphisms, $\varphi$ is also a vector space isomorphism. Note that $\varphi\alpha(x+i)=\varphi(\alpha(x)+\alpha(i))=\overline{\alpha(x)}+q_{\mathfrak{g}}^{*}(\alpha(i))=\overline{\alpha(x)}+q_{\mathfrak{g}}^{*}(i)\bar{\alpha}
 =(\bar{\alpha}+\tilde{\bar{\alpha}})(\bar{x}+q_{\mathfrak{g}}^{*}(i))=(\bar{\alpha}+\tilde{\bar{\alpha}})\varphi(x+i)$, i.e. $\varphi\alpha=(\bar{\alpha}+\tilde{\bar{\alpha}})\varphi$. By the same way, $\varphi\beta=(\bar{\beta}+\tilde{\bar{\beta}})\varphi$. Furthermore, $\forall\, x,y,z\in \mathfrak{g}$, $i,j,k\in I$,

$\begin{array}{lllll}
&&\varphi([\beta(x+i),\beta(y+j),\alpha(z+k)])\\
&=&\varphi([\beta(x)+\beta(i),\beta(y)+\beta(j),\alpha(z)+\alpha(k)])\\
&=&\varphi([\beta(x),\beta(y),\alpha(z)]+[\beta(x),\beta(y),\alpha(k)]+[\beta(x),\beta(j),\alpha(z)]+[\beta(x),\beta(j),\alpha(k)]\\
&&+[\beta(i),\beta(y),\alpha(z)]
+[\beta(i),\beta(y),\alpha(k)]+[\beta(i),\beta(j),\alpha(z)]+[\beta(i),\beta(j),\alpha(k)])\\
&=&\varphi([\beta(x),\beta(y),\alpha(z)]+[\beta(x),\beta(y),\alpha(k)]+[\beta(x),\beta(j),\alpha(z)]+[\beta(i),\beta(y),\alpha(z)])\\
&=&\varphi(p_{0}([\beta(x),\beta(y),\alpha(z)])\!+\!p_{1}([\beta(x),\beta(y),\alpha(z)])\!+\![\beta(x),\beta(y),\alpha(k)]\!+\![\beta(x),\beta(j),\alpha(z)]\\
&&+[\beta(i),\beta(y),\alpha(z)])\\
&=&\overline{[\beta(x),\beta(y),\alpha(z)]}+q^*_\mathfrak{g}\big(p_{1}([\beta(x),\beta(y),\alpha(z)])+[\beta(x),\beta(y),\alpha(k)]+[\beta(x),\beta(j),\alpha(z)]\\
&&+[\beta(i),\beta(y),\alpha(z)]\big)
\end{array}$
$\begin{array}{lllll}
&=&\overline{[\beta(x),\beta(y),\alpha(z)]}+\theta(\overline{\beta(x)},\overline{\beta(y)},\overline{\alpha(z)})+\ad^*(\overline{\beta(x)},\overline{\beta(y)})q_{\mathfrak{g}}^{*}(\alpha(k))\\
&&-(-1)^{|j||z|}\ad^*(\overline{\beta(x)},\overline{\beta(z)})q_{\mathfrak{g}}^{*}(\alpha(j))+(-1)^{|i|(|y|+|z|)}\ad^*(\overline{\beta(y)},\overline{\beta(z)})q_{\mathfrak{g}}^{*}(\alpha(i))\\
&=&[\overline{\beta(x)}+q_{\mathfrak{g}}^{*}(\beta(i)),\overline{\beta(y)}+q_{\mathfrak{g}}^{*}(\beta(j)),\overline{\alpha(z)}+q_{\mathfrak{g}}^{*}(\alpha(k))]_\theta\\
&=&[\varphi(\beta(x)+\beta(i)),\varphi(\beta(y)+\beta(j)),\varphi(\alpha(z)+\alpha(k))]_\theta.
\end{array}$
Then $\varphi$ is an isomorphism of superalgebras, and $(B\oplus B^{*},[\cdot,\cdot,\cdot]_\theta, \bar{\alpha}+\tilde{\bar{\alpha}},\bar{\beta}+\tilde{\bar{\beta}})$ is a $3$-Bihom-Lie superalgebra.
Furthermore, we have
{\setlength\arraycolsep{2pt}
\begin{eqnarray*}
q_{B}(\varphi(x+i),\varphi(y+j))&=&q_{B}(\bar{x}+q_{\mathfrak{g}}^{*}(i),\bar{y}+q_{\mathfrak{g}}^{*}(j))\\
&=&q_{\mathfrak{g}}^{*}(i)(\bar{y})+q_{\mathfrak{g}}^{*}(j)(\bar{x})\\
&=&q_{\mathfrak{g}}(i,y)+q_{\mathfrak{g}}(j,x)\\
&=&q_{\mathfrak{g}}(x+i,y+j),
\end{eqnarray*}}
then $\varphi$ is isometric. And $\forall\, x,y,z,w\in \mathfrak{g}$, the relation
\begin{eqnarray*}
&&q_{B}([(\bar{\beta}+\tilde{\bar{\beta}})(\varphi(x)),(\bar{\beta}+\tilde{\bar{\beta}})(\varphi(y)),(\bar{\alpha}+\tilde{\bar{\alpha}})(\varphi(z))]_\theta,(\bar{\alpha}+\tilde{\bar{\alpha}})(\varphi(w)))\\
&=&q_{B}([\varphi(\beta(x)),\varphi(\beta(y)),\varphi(\alpha(z))]_\theta,\varphi(\alpha(w)))=q_B(\varphi([\beta(x),\beta(y),\alpha(z)]),\varphi(\alpha(w)))\\
&=&q_{\mathfrak{g}}([\beta(x),\beta(y),\alpha(z)],\alpha(w))=-(-1)^{|z|(|x|+|y|)}q_{\mathfrak{g}}(\alpha(z),[\beta(x),\beta(y),\alpha(w)])\\
&=&-(-1)^{|z|(|x|+|y|)}q_{B}(\varphi(\alpha(z)),[\varphi(\beta(x)),\varphi(\beta(y)),\varphi(\alpha(w))]_\theta)\\
&=&-(-1)^{|z|(|x|+|y|)}q_{B}((\bar{\beta}+\tilde{\bar{\beta}})(\varphi(z)),[(\bar{\beta}+\tilde{\bar{\beta}})(\varphi(x)),(\bar{\beta}+\tilde{\bar{\beta}})(\varphi(y)),(\bar{\alpha}+\tilde{\bar{\alpha}})(\varphi(w))]_\theta),
\end{eqnarray*}
 which implies that $q_B$ is $\alpha\beta$-invariant. So $(B\oplus B^{*}, q_B,\bar{\beta}+\tilde{\bar{\beta}},\bar{\alpha}+\tilde{\bar{\alpha}})$ is a quadratic $3$-Bihom-Lie superalgebra.
Thus, the $T_\theta^*$-extension $(B\oplus B^{*}, q_B,\bar{\beta}+\tilde{\bar{\beta}},\bar{\alpha}+\tilde{\bar{\alpha}})$ of $B$ is isometric to $(\mathfrak{g},q_{\mathfrak{g}},\alpha,\beta)$.
\end{proof}

%%%%%%%%%%%%%%%%%%%%%%%%%%%%%%%%%%%%%%%%%%%%%%%%%%

%%%%%%%%%%%%%%%%%%%%%%%%%%%%%%%%%%%%%%%%%%%%%%%%

%\section*{References}

\end{document}